\documentclass[12pt]{amsart}
\usepackage{verbatim}
\usepackage{graphicx}
\usepackage{amssymb}
\usepackage{amsfonts}
\usepackage{amsmath}
\usepackage{hyperref}

\usepackage{color}

\def\beq{\begin{equation}}
\def\eeq{\end{equation}}

\def\ben{\begin{enumerate}}
\def\een{\end{enumerate}}

\def\RR{\mathbb{R}}
\def\CC{\mathbb{C}}
\def\NN{\mathbb{N}}

\def\eps{\varepsilon}

\def\cone{\mathrm{cone}}
\def\id{\mathrm{Id}}

\def\pt{\mathrm{pt}}

\def\J{\mathcal{I}}
\def\ev{\mathrm{ev}}

\def\cC{ {\mathcal C} }

\def\cE{ {\mathcal E} }
\def\cF{\Pi}

\def\cI{ {\mathcal I} }

\def\cM{{\mathcal M}}

\def\cR{{ \mathcal R }}
\def\cS{{\mathcal S} }

\def\cW{{\mathcal W}}

\def\lnss{left nullstellensatz property}

\def\qrr{real radical}
\def\rad{saturation}

\def\chrisT{\ast}
\def\FNSx{F \langle x,x^{\chrisT} \rangle}

\def\chrisA{\FNSx}

\newtheorem{theorem}{Theorem}[section]
\newtheorem{lemma}[theorem]{Lemma}
\newtheorem{thm}[theorem]{Theorem}
\newtheorem{prop}[theorem]{Proposition}
\newtheorem{cor}[theorem]{Corollary}
\theoremstyle{definition}
\newtheorem{remark}[theorem]{Remark}
\newtheorem{exa}[theorem]{Example}

\newtheorem{definition}[theorem]{Definition}

\def\bs{\bigskip}

\def\chrisA{\mathcal A}
\def\FA{\mathfrak F}

\newcommand{\rr}[1]{\sqrt[\mathrm{rr}]{#1}}
\def\RradI{\sqrt[\cR]{I}   }

\numberwithin{equation}{section}

\begin{document}

\title[A NC \ real nullstellensatz]{A Non-commutative Real Nullstellensatz Corresponds to a  Non-commutative Real Ideal; Algorithms}

\author[Cimpri\v c]{Jakob Cimpri\v c${}^1$}
\address{Jakob Cimpri\v c, Department of Mathematics, University of Ljubljana, Jadranska 21, SI-1000 Ljubljana, Slovenia}
\email{cimpric@fmf.uni-lj.si}
\urladdr{http://www.fmf.uni-lj.si/$\sim$cimpric/}
\thanks{${}^1$Research supported by the grant P1--0222 from the Slovenian Research Agency – ARRS}

\author[Helton]{J. William Helton${}^2$}
\address{J. William Helton, Department of Mathematics, University of California San Diego, 9500 Gilman Drive, La Jolla, California 92093-0112, USA}
\email{helton@math.ucsd.edu}
\urladdr{http://math.ucsd.edu/$\sim$helton/}
\thanks{${}^2$Research supported by NSF grants
DMS-0700758, DMS-0757212, and the Ford Motor Co.}

\author[McCullough]{Scott McCullough${}^3$}
\address{Scott McCullough, Department of Mathematics, University of Florida, 490 Little Hall, Gainesville, Florida 32611-8105, USA}
\email{sam@math.ufl.edu}
\urladdr{http://www.math.ufl.edu/$\sim$sam/}
\thanks{${}^3$Research supported by the NSF grant DMS-0758306.}

\author[Nelson]{Christopher Nelson${}^{2\dag}$}
\address{Christopher Nelson, Department of Mathematics, University of California San Diego, 9500 Gilman Drive,
La Jolla, California 92093-0112, USA}
\email{csnelson@math.ucsd.edu}
\thanks{${}^\dag$Much of the  material in this paper is part of the Ph.D. thesis of Christopher Nelson at UCSD}




\begin{abstract} 

 This article extends the classical Real Nullstellensatz of Dubois and Risler
 to left ideals in free $*-$algebras $\RR \langle x,x^\ast \rangle$ with
 $x=(x_1,\ldots,x_n)$.

 First we introduce notions of the (noncommutative)
 zero set of a left ideal  and of a real left ideal.
  We prove that every element from $\RR \langle x,x^\ast \rangle$
  whose zero set contains the intersection of zero sets of elements
  from a finite subset $S$ of $\RR \langle x,x^\ast \rangle$
 belongs to the smallest real left ideal  containing $S$.

 Next we give an implementable algorithm which
 for every finite $S \subset \RR \langle x,x^\ast \rangle$
 computes the smallest real left ideal containing $S$
 and prove that the algorithm succeeds
 in a finite number of steps.

 Our definitions and some of our results also work for
 other $*$-algebras. As an example we treat real
 left ideals  in  $M_n(\RR[x_1]).$

\end{abstract}

\maketitle

\section{Introduction}
  This article establishes analogs, in the setting of (some)  $*$-algebras,
  of the classical real Nullstellensatz of Dubois and Risler.  Accordingly,
  to state results, it is first necessary to discuss both noncommutative
  zero sets and real ideals and radicals.
  These topics are treated below in Subsections \ref{subsec:zeros} 
  and \ref{subsec:rad} respectively. The introduction concludes
  with a brief discussion of the main results in Subsection 
  \ref{subsec:results}. 
  
  Our approach to Noncommutative Real Algebraic Geometry is motivated
  by \cite{hmp}; for alternative approaches see \cite{sch2} and \cite{mm}.

\subsection{Zero Sets in $*$-Algebras}
 \label{subsec:zeros}
Let $F$ be either $\RR$ or $\CC$ with complex conjugation as involution.
Let $\chrisA$ be a unital associative $F$-algebra with involution $\ast$,
or \textit{$\ast$-algebra} for short.
Let $V$ be a pre-Hilbert space, i.e. an $F$-vector space with an inner
product.  A mapping $\pi$ of $\chrisA$ into the set of $F$-linear operators defined on $V$
is said to be a (unital) $\ast$-\textit{representation} of $\chrisA$
on $V$ if $\pi(1)=1$ and
it satisfies the familiar axioms:
\[
 \begin{split}
  \pi(\alpha_1 a_1+\alpha_2 a_2) v = &\alpha_1 \pi(a_1) v+\alpha_2
  \pi(a_2) v \\
\pi(a_1 a_2)v=&\pi(a_1) \pi(a_2) v \\
\langle \pi(a) v_1,v_2 \rangle = &
\langle v_1,\pi(a^\ast) v_2 \rangle
\end{split}
\]
 for every $a,a_1,a_2 \in \chrisA$, $\alpha_1,\alpha_2 \in F$ and $v,v_1,v_2 \in V$.

Let $\cR$ be the {\it class of all
$\ast$-representations }  of the $\ast$-algebra $\chrisA$.
Usually, we are only interested in some subclass of ``well-behaved"
$\ast$-representations, such as the subclass $\cF$
of {\it all finite-dimensional \linebreak
$\ast$-representations}. In the following let $\cC$
be a fixed subclass of $\cR$.

A $\cC$-\textit{point} of $\chrisA$ is an ordered pair $(\pi,v)$
such that $\pi \in \cC$ and $v \in V_\pi$.
Write $\pt_{\cC}(\chrisA)$ for the
set of all $\cC$-points of the $\ast$-algebra $\chrisA$.
For every subset $S$ of $\chrisA$ write
\[
V_{\cC}(S):=\{(\pi,v) \in \pt_{\cC}(\chrisA) \mid \pi(s)v=0 \mbox{ for every } s \in S\}.
\]
Clearly, $V_{\cC}(S)=V_{\cR}(S) \cap \pt_{\cC}(\chrisA)$.
For a subset $T$ of $\pt_{\cR}(\chrisA)$, let
\[
\J(T) := \{a \in \chrisA \mid \pi(a)v=0 \mbox{ for every } (\pi,v) \in T\}.
\]
Note that $\J(T)$ is always a left ideal.

Now we give three examples.

\begin{exa}
\label{ex:freeAlg}
  Let $\FA=F\langle x,x^* \rangle$ denote the
  {\bf free $\ast$-algebra}
   on  $x=(x_1, \cdots, x_g)$.
 Given a $g$-tuple $X=(X_1,\dots,X_g)$ of same size square matrices over $F$, write
 $\pi_X(p):= p(X)$, where $p(X)$ is the natural evaluation of $p$ at
 $X$.  It is evident that $\pi_X$ is a $\ast$-representation
 of $\FA$ on the  Hilbert space $F^N$ ($N$ is the size of $X$) and is thus an element of the class $\cF$.
 Conversely,  every element $\pi$ of  $\cF$ is equal to $\pi_X$ for some $g$-tuple $X$ (namely $X_j=\pi(x_j)$).
 Therefore, the $\cF$-points of $\FA$ can be identified with pairs $(X,v)$
 with $v$ being in $F^N$.
 For  $S \subset \FA$ we have
\[
V_\cF(S)= \{ (X,v) \mid \ p(X)v=0 \ \mbox{ for every }  \ p \in S \}.
\]
For a subset $T$ of $\pt_{\cF}(\FA)$ we have
\[
\J(T) = \{p \in \FA \mid  p(X)v=0 \mbox{ for every } (X,v) \in T \}.
\]
 As we shall see,  in the case of $\FA$, for many purposes $\cF$ is a well-behaved
  subclass of $\cR$. \qed
\end{exa}

\begin{exa}
\label{ex:matpoly}
Let $F[x]$ denote the algebra of all polynomials in variables $x=(x_1, \cdots, x_g)$
with coefficients from $F \in \{\RR,\CC\}$.
For every $n$, let $M_n(F[x])$ denote the algebra of $n \times n$ matrices
with entries in $F[x]$.
The involution $\bar{ }$ on $F[x]$ conjugates the coefficients
and the involution $\ast$ on $M_n(F[x])$
is the conjugated transpose, i.e. $[p_{ij}]^\ast =[\overline{p_{ji}}]$.

For every point $a \in \RR^g$ its evaluation mapping $\ev_a \colon M_n(F[x]) \to
M_n(F)$ defined by $\ev_a([p_{ij}]):=[p_{ij}(a)]$, is a $\ast$-representation of $M_n(F[x])$
on $F^n$. (The evaluations in complex points need not be $\ast$-representations.)
The class $\cE:=\{\ev_a \mid a \in \RR^g\}$ is a proper subclass of $\cF$.
Note that the $\cE$-points of $M_n(F[x])$
can be identified with pairs $(a,v)$ where $a \in \RR^g$ and $v \in F^n$,
i.e. $\pt_{\cE}(M_n(F[x]))=\RR^g \times F^n$.
For $S \subset M_n(F[x])$ we have
\[
V_\cE(S)= \{ (a,v) \in \RR^g \times F^n  \mid p(a)v=0 \ \mbox{ for every }  \ p \in S \}.
\]
For a subset $T$ of $\RR^g \times F^n$ we have
\[
\J(T) = \{p \in M_n(F[x]) \mid  p(a)v=0 \mbox{ for every } (a,v) \in T \}.
\]
This example also makes sense for $g=0$. In this case $F[x]=F$, so that $M_n(F[x])=M_n(F)$.
Moreover, $\RR^g=\{0\}$, so the only element of $\cE$ is $\id \colon M_n(F) \to M_n(F)$.
\qed
\end{exa}

\begin{exa}
The polynomial algebra $F[y]$, $y=(y_1,\ldots,y_g)$, $F \in \{\RR,\CC\}$,
with involution $y_i^\ast=-y_i$ for $i=1,\ldots,g$ and $\alpha^\ast=\bar{\alpha}$ for $\alpha \in F$ has
a natural $\ast$-representation $\pi_0$ acting
on the Schwartz space $\cS(\RR^g,F)$ of rapidly decreasing functions.
It assigns to each $y_i$ the partial derivative $\frac{\partial}{\partial t_i}$ so each $\pi_0(p)$ is the partial differential operator $p(D)$.
The set of $\{\pi_0\}$-points is $\pt_{\{\pi_0\}}(\cW_g)=\{\pi_0\} \times \cS(\RR^g,F)$
which can be identified with $\cS(\RR^g,F)$.
For every $S \subseteq \RR[y]$ we have
\[
V_{\{\pi_0\}}(S)= \{ \psi \in  \cS(\RR^g,F) \mid \pi_0(p) \psi=0 \ \mbox{ for every }  \ p \in S \}
\]
which is the set of all solutions of the partial differential equations from $S$.
For a subset $T$ of $\cS(\RR^g,F)$ we have
\[
\J(T) = \{p \in \RR[y] \mid  \pi_0(p)\psi=0 \mbox{ for every } \psi \in T \}
\]
which is the set of all partial differential equations whose solution sets contain $T$.
(We will not discuss this example in other sections but
see \cite{ss} for a Nullstellensatz in the spirit of this paper.
The definitions can also be extended to partial differential equations with non-constant
coefficients but we are not aware of any results in this direction.) \qed
\end{exa}

\subsection{Radicals and Noncommutative Real Ideals}
 \label{subsec:rad}
  For a left ideal $I$ of $\chrisA$ and a class 
  $\cC$ of $\ast$-representations of $\chrisA$, we call the radical
\[
\sqrt[\cC]{I} := \J(V_{\cC}(I)) 
\]
the \textit{$\cC$-{\rad}} of $I$.
 Evidently $\sqrt[\cC]{I}$ is a left ideal. 
We say that $I$ \textit{has the  \lnss \ for $\cC$-points}
if $\sqrt[\cC]{I}=I$.
Lemma \ref{lem:cradicals} lists the basic facts.

\begin{lemma}
 \label{lem:cradicals}
 Let $\cC$ be a representation class and $I$ a left ideal of $\chrisA$.

 The radical $\sqrt[\cC]{I}$ is the smallest left ideal
 which contains $I$ and has the \lnss \ for $\cC$-points.

 For every subset $S$ of $\chrisA$, $V_{\cC}(S)=V_{\cC}(I_S)=V_{\cC}(\sqrt[\cC]{I_S})$
 where $I_S$ is the left ideal of $\chrisA$ generated by $S$.

 If $I \subseteq I'$ then $\sqrt[\cC]{I}\subseteq \sqrt[\cC]{I'}$.
 If $\cC \subseteq \cC'$ then $\sqrt[\cC']{I} \subseteq \sqrt[\cC]{I}$.
 
 For every subset $T$ of $\pt_{\cC}(\chrisA)$ we have that $\sqrt[\cC]{\J(T)}=\J(T)$.
\end{lemma}

\begin{proof}
All claims are straightforward consequences of the following properties:
\begin{enumerate}
\item[(a)] if $\cC \subseteq \cC'$ then $V_{\cC}(S) \subseteq V_{\cC'}(S)$,
\item[(b)] if $S \subseteq S'$ then $V_{\cC}(S') \subseteq V_{\cC}(S)$,
 \item[(c)] if $T \subseteq T'$ then $\J(T') \subseteq \J(T)$,
\item[(d)] $S \subseteq \J(V_{\cC}(S))$,
\item[(e)] $T \subseteq V_{\cC}(\J(T))$.
\end{enumerate}
\end{proof}

In addition to shedding light on the basic question of
which ideals have the \lnss \ for $\cC$-points, we would
also like to find an algebraic description of the $\cC$-{\rad}
similar to the notion of real radical in classical
real algebraic geometry,
see \cite[Definition 6.4 and Theorems 6.5 and 6.7]{lam}
or Example \ref{ex:classical} below.

These considerations motivate the following definitions.
A left ideal $I$ of $\chrisA$ is said to be \textit{real}
if for every $a_1,\ldots,a_r$ of $\chrisA$
such that
\[
\sum_{i=1}^r a_i^\ast a_i \in I+I^\ast ,
\]
we have that $a_1,\ldots,a_r \in I$.
An intersection of a family of real ideals is a real ideal.
For a left ideal $J$ of $\chrisA$ we call the ideal
\[
\rr{J} = \bigcap_{I \supseteq J,I \text{ real}} I
=\text{the smallest real ideal containing } J
\]
\textit{the {\qrr}} of $J$.
Here are the basic properties.

\begin{lemma}
 \label{lem:qradicals}
Let $\cC$ be a representation class and $I$ a left ideal of $\chrisA$.

If $I$  has the \lnss \ for $\cC$-points, then $I$ is a real ideal.

The $\cC$-{\rad} of $I$ contains the {\qrr} of $I$.
\end{lemma}

\begin{proof}
  To prove the first claim,
  suppose $I$ has the \lnss,  each of $a_1,\dots,a_r$ are in $\chrisA,$
   $b,c$ are  in $I$
  and $\sum a_j^* a_j =b +c^*$.  Let $(\pi,v)\in\cC$ be given.  In
  particular, $\pi(b)v=0=\pi(c)v.$ Thus,
\[
\begin{split}
\sum \langle \pi(a_j)v,\pi(a_j)v\rangle
 = & \sum \langle \pi(a_j^* a_j)v,v\rangle \\
 = & \langle \pi(b)v,v\rangle + \langle v,\pi(c)v\rangle \\
 = & 0.
\end{split}
\]
It follows that $\pi(a_j)v=0$ and therefore $a_j \in  \cI(V_{\cR}(I)).$
Hence, by the \lnss, $a_j\in I$ and $I$ is a real ideal.

To prove the second claim note that the first claim implies that
the smallest left ideal which contains $I$ and has \lnss \ for $\cC$-points
contains the smallest real left ideal which contains $I$.
Now use the first claim of Lemma \ref{lem:cradicals}
and the definition of the {\qrr} to finish the proof.
\end{proof}

Lemmas \ref{lem:cradicals} and \ref{lem:qradicals} imply that
\begin{equation*}
   I \subseteq \rr{I} \subseteq \sqrt[\cR]{I}
  \subseteq \sqrt[\cC]{I}
\end{equation*}
for  every representation class $\cC$ and every left ideal $I$ of $\chrisA$.

\subsection{Summary of Results}
 \label{subsec:results}

The main result of this paper is

\begin{thm} 
\label{thm:main}
\mbox{} \\
 A finitely generated left ideal $I$ in $\FNSx$
  satisfies the 
\lnss \ for $\Pi$-points
if and only if $I$ is real. Moreover, 
\[
  I  \subseteq  \rr{I} = \sqrt[\cR]{I}
  = \sqrt[\Pi]{I}.
\]
\end{thm}

In Section \ref{sec:prepthy}, we prove several technical results about the $\ast$-algebra $\FNSx$
which are similar to Gr\" obner bases computations. 

In Section \ref{sec:alg} we present an (implementable and effective) algorithm for computing
the real radical of a finitely generated left ideal in $\FNSx$. Its theoretical
importantance is in the fact that the result is always a finitely generated left ideal.
Therefore, the second part of Theorem \ref{thm:main} follows from the first.

The first part of Theorem \ref{thm:main} is proved in Section \ref{sec:pfmain}. The idea
is to show that every finitely generated real ideal in $\FNSx$ 
is of the form $\{a \in \FNSx \mid L(a^\ast a)=0\}$ for some positive functional $L$.

In Section \ref{sec:sqrtRI} we shift our attention to general $\ast$-algebras.
We prove a topological characterization of the $\cR$-saturation and develop
a (non-effective) iterative procedure for computing the real radical.

In Section \ref{sec:examples} we prove that all left ideals 
 $I$ in $M_n(F[x_1])$ satisfy $I \subseteq \rr{I} = \sqrt[\cR]{I} = \sqrt[\cE]{I}$.
The case of several variables remains open.

\section{Ideals and their Complements}
\label{sec:prepthy}

In this section we prove  a collection of basic facts which constitute the backbone of the main results of this paper.
We begin by stating an appealing theorem, Theorem \ref{thm:lowdeg}, 
which underlies the success of our algorithm given in \S \ref{sec:alg}.
In the course of its proof we lay out essentials for our main theorem. 
Recall that $F$  is $\RR$ or $\CC$ and
 $\FA=F\langle x,x^*\rangle.$

\begin{definition}
\label{def:FAd}
Let $\FA_{d}$ be the vector
space spanned by all polynomials in
$\FA$ with degree bounded by $d$.
In general, given a vector subspace $V \subseteq \FA$, $V_{d}$
denotes the space
of elements of $V$ with degree bounded by $d$.
\end{definition}

\begin{exa}If $V = \FA x_1x_1$,
then $V_{3}$ is the space
\[V_3 = \mathrm{span}\{x_1x_1x_1, x_1^*x_1x_1, x_2x_1x_1, x_2^*x_1x_1, x_1x_1\}.\]
\qed
\end{exa}

\begin{exa}Let $x = (x_1, x_2)$ and let $W = \FA(x_1x_1 + 1)$. Each
  element of $W$ is of the form
$a(x_1x_1 + 1)$ for some $a \in \FA$. If $a$ is nonzero, then the
degree of $a(x_1x_1 + 1)$
is equal to $2 + \deg(a)$.  Therefore all elements of $W$ of degree bounded by $3$ are of the form
\[W_3 = \{a(x_1x_1 + 1):\ \deg(a) \leq 1 \}. \]
Therefore $W_3$ is the spanned by the basis
\[\{x_1(x_1x_1+1), x_1^*(x_1x_1+1), x_2(x_1x_1+1), x_2^*(x_1x_1+1), x_1x_1+1\}. \]
\qed
\end{exa}

\begin{definition}
\label{def:oplus}
Let $V$ be a vector space and let $W_1$ and $W_2$ be vector subspaces of $V$.
If $W_1 \cap W_2 = (0)$, let $W_1 \oplus W_2$ denote the space $W_1 + W_2 \subseteq V$.
If $W_1 \cap W_2 \supsetneq (0)$, then
$W_1 \oplus W_2$ is undefined.
\end{definition}

A main result of this section is

\begin{thm}
\label{thm:lowdeg}
Let $I \subseteq \FA$ be a finitely-generated
left ideal.
Suppose $I$ is generated by polynomials
$p_1, \ldots, p_k \in \FA$ with $\deg p_i \leq d$ for each $i$. 
Then the following are equivalent.

\begin{enumerate}
\item $I$ is a real ideal.

\item If  $q_1, \ldots, q_k$ are polynomials and
 $\sum_{i=1}^{\ell}q_i^*q_i \in I + I^*,$ then $q_j\in I$ for each
 $j$.

\item If $V$ is a subspace of $\FA_{d-1}$ such that
\[\FA_{d-1} = I_{d-1} \oplus V\]
and $v_j\in V$ are polynomials such that
$\sum_{i=1}^{\ell} v_i^*v_i \in I + I^*,$ then each $v_j=0$.
\end{enumerate}
\end{thm}

The proof of this theorem appears in $\S$ \ref{subsec:pf}.

An important corollary to Theorem \ref{thm:lowdeg} is the following.

\begin{cor}
Let $I \subseteq \FA$ be a finitely-generated
left ideal.
Suppose $I$ is generated by polynomials
$p_1, \ldots, p_k \in \FA$ with $\deg p_i \leq d$ for each $i$. 
Then $I$ is real if and only if whenever
\[\sum_{i=1}^{\ell} q_i^*q_i \in I + I^*, \quad \deg(q_1), \ldots, \deg(q_{\ell}) < d, \] 
then $q_1, \ldots, q_{\ell} \in I$.
\end{cor}

\begin{proof}
 Suppose $q_1, \ldots, q_{\ell}$ have degree less than
$d$ and that
\[\sum_{i=1}^{\ell} q_i^*q_i \in I + I^*.\]
Decompose $\FA_{d-1}$ as
\[\FA_{d-1} = I_{d-1} \oplus V\]
and express each $q_i$ as
\[ q_i = q_{i,I} + q_{i,V}, \quad \quad q_{i,I} \in I_{d-1}, \quad q_{i,V} \in V.\]
Then
\[ \sum_{i=1}^{\ell} q_i^*q_i = \sum_{i=1}^{\ell} \left( q_{i,I}^*q_i + q_{i,V}^*q_{i,I} + q_{i,V}^*q_{i,V}\right) \in I + I^*,\]
which implies that
\begin{equation}
\label{eq:sumQiV}
 \sum_{i=1}^{\ell} q_{i,V}^*q_{i,V} \in I + I^*.
\end{equation}
By Theorem \ref{thm:lowdeg}, $I$ is real 
if and only if (\ref{eq:sumQiV}) implies
that $q_{i,V} = 0$ for each $i$.
However, each $q_{i} \in I$ if and only if
$q_{i,V} = 0$. This proves the corollary.
\end{proof}


\subsection{Proof and Further Facts}
\label{subsec:pf}


We give a string of facts, which typically involve complements 
and the degree of polynomials,
that underlie  proofs of Theorem \ref{thm:lowdeg}.

\begin{definition}
\label{def:FAHd}Let $\FA^H_d$ denote the vector space
of all homogeneous degree $d$ polynomials in
$\FA$. ($0$ is considered homogeneous of all degrees.)
In general, given a vector subspace
$V \subseteq \FA$, $V^H_d$ denotes the space $V \cap \FA^H_d$
of all homogeneous degree $d$ elements of $V$.
\end{definition}

\begin{exa}
\label{exa:nmo1}
Let $x = (x_1, x_2)$ so that
$\FA = F\langle x_1, x_2, x_1^*, x_2^*\rangle$.
If $V = \FA x_1x_1$, then $V^H_3$ is the space
\[V^H_3 = \mathrm{span}\{x_1x_1x_1, x_1^*x_1x_1, x_2x_1x_1, x_2^*x_1x_1\}.\]
\qed \end{exa}

\begin{definition}
\label{def:leadingpoly}For each nonzero $p \in \FA$, the \textbf{leading polynomial} of $p$
is the unique homogeneous polynomial $p'$ such that $\deg(p) = \deg(p')$ and $\deg(p-p') < \deg(p)$.
For a space $V \subset \FA$, let $V_d^{\ell}$ denote the space spanned by the leading polynomials of
all degree $d$ elements of $V$.  Note that $V_d^{\ell}$ is contained in the space $\FA_d^H$.
\end{definition}

\begin{exa}Let $x = (x_1, x_2)$ and let $I = \FA (x_1x_1 + 1) + \FA x_2$. Then $I_2$ is the space
\[I_2 = \mathrm{span}\{x_1x_1 + 1, x_1x_2, x_1^*x_2, x_2x_2, x_2^*x_2, x_2\}.\]
The space spanned by all homogeneous degree two polynomials is
\[I_2^H = \mathrm{span}\{x_1x_2, x_1^*x_2, x_2x_2, x_2^*x_2 \}.\]
The leading polynomial of $x_1x_1 + 1$ is $x_1x_1$ and the leading polynomial of each $zx_2$ is itself, $zx_2$,
where $z = x_1, x_1^*, x_2,$ or $x_2^*$. It follows that
\[I^{\ell}_2 = \mathrm{span}\{x_1x_1, x_1x_2, x_1^*x_2, x_2x_2, x_2^*x_2\}.\]
\qed \end{exa}


\begin{definition}
For every pair of subsets $A$ and $B$ of $\FA$ we write $AB$ for the set of all
finite sums of elements of the form $ab$, $a \in A$, $b \in B$.
\end{definition}

\begin{exa}
\label{ABdef}
Clearly, $\FA_k^H \FA_l^H=\FA_{k+l}^H$ for every $k$ and $l$.
If $\FA_l^H=U \oplus V$ for some vector spaces
$U$ and $V$, then $\FA_k^H \FA_l^H=\FA_k^H U \oplus \FA_k^H V$
(since $\FA_k^H U \cap \FA_k^H V=\{0\}$
by Lemma \ref{eq:LinIndOfHomNcPolys}.)
\qed
\end{exa}

\begin{lemma}
\label{eq:LinIndOfHomNcPolys}Let $p_1, \ldots, p_k \in \FA$ be linearly independent, homogeneous degree $d$ polynomials.
Then
\[q_1 p_1 + \ldots + q_k p_k = 0 \]
for some polynomials $q_1, \ldots, q_k \in \FA$ if and only if each $q_i = 0$.
\end{lemma}

\begin{proof}Suppose
\[q_1 p_1 + \ldots + q_k p_k = 0 \]
for some polynomials $q_1, \ldots, q_k \in \FA$.
Let $\cM$ be a finite set of monomials such that there exist scalars
$A_{m,i}$, for $i=1, \ldots, k$, such that
\[q_i = \sum_{m \in \cM} A_{m,i} m.\]
For each $m \in \cM$,
\[
r_m = \sum_{i=1}^k A_{m,i} p_i
\]
is a homogeneous polynomial of degree $d$. Since
\[
\sum_{m \in \cM} m r_m =\sum_{i=1}^k q_i p_i= 0,
\]
it follows that $m r_m=0$ for all $m \in \cM$. (This is true because if $m_1 \ne m_2 \in \cM$ then
$m_1 r_{m_1}$ and $m_2 r_{m_2}$ have disjoint monomials. This in turn is true for the following reason: if $\deg m_1 \ne \deg m_2$
then they have monomials with different degrees; if $\deg m_1 = \deg m_2$
then they have monomials with different initial words.)
Since all $r_m$ are $0$ and the $p_i$ are linearly independent, all $A_{m,i}$ must be $0$.
\end{proof}

\begin{lemma}
\label{eq:LinIndOfHomNcPolys2}
Let $p_1, \ldots, p_k \in \FA$ be  degree $d$ polynomials with linearly independent leading polynomials $p_1',\ldots,p_k'$.
For every $q_1, \ldots, q_k \in \FA$ such that at least one $q_i$ is nonzero and for every $u \in \FA_{d-1}$,
the element
\[q=\sum_{i=1}^k q_i p_i+u\]
is nonzero, has degree $d+e$ where $e=\max\{\deg (q_i) \mid i=1,\ldots,k\}$
and its leading polynomial is $q'=\sum_{\deg (q_i)=e} q_i' p_i'$.
\end{lemma}

\begin{proof}Suppose that at least one $q_i$ is nonzero.
Let $e = \max_i\{\deg(q_i)\}$. Let $\hat{q'}_i = q'_i$ if $\deg(q_i) = e$
and let $\hat{q'}_i = 0$ otherwise.
Then
\begin{equation}
\label{eq:degreenpluslower}
q = \sum_{i=1}^k \hat{q'}_ip'_i + \sum_{i=1}^k (q_i-\hat{q'}_i)p_i + \sum_{i=1}^k \hat{q'}_i(p_i-p'_i) + u.
\end{equation}
By linear independence of the $p'_i$ and by Lemma \ref{eq:LinIndOfHomNcPolys}, the homogeneous polynomial
$\displaystyle\sum_{i=1}^k \hat{q'}_ip'_i$
can only be zero if all of the $\hat{q'}_i$ equal $0$, which cannot be. Further, each of the
other terms of (\ref{eq:degreenpluslower}) must be of degree less than $d+e$.
Therefore, the leading polynomial of $q$ is
\[q' = \sum_{i=1}^k \hat{q'}_ip'_i.\]
\end{proof}


\begin{lemma}
\label{lem:Adl}Let $I \subseteq \FA$ be a left ideal generated by 
  polynomials of degree bounded by $d$.
\begin{enumerate}
\item \label{it:Adl1}
 There exist $p_1, \ldots, p_k \in I$ such that $\deg(p_i) = d$ for each $i$, the
leading polynomials $p'_1, \ldots, p'_k$ are linearly independent, and $I$ is equal to
\[I = \bigoplus_{i=1}^k \FA p_i \oplus I_{d-1}. \]

\item \label{it:Adl2}
  For each $D \geq d$, the space $I_{D}^{\ell}$ is equal to
\[I_D^{\ell} =  \sum_{i=1}^k \FA_{D-d}^H p'_i.\]
\end{enumerate}
\end{lemma}

\begin{proof}
First, $I$ being generated by polynomials of degree bounded by $d$ implies that $I = \FA I_d$.
  To prove item \eqref{it:Adl1}, let $p_1, \ldots, p_k \in I$ be a maximal set of degree $d$ polynomials in $I$ such that
the leading polynomials $p'_1, \ldots, p'_k$ are linearly independent. 

By Lemma \ref{eq:LinIndOfHomNcPolys},
for any $a_1, \ldots, a_k \in \FA$, not all equal to $0$,
we have $\sum_{i=1}^k a_i p_k' \neq 0$.
Therefore
$\sum_{i=1}^k a_i p_k \neq 0$, so 
\[ \sum_{i=1}^k \FA p_i =  \bigoplus_{i=1}^k \FA p_i. \]  
Further note that each $\sum_{i=1}^k a_i p_k \neq 0$
must have degree at least $d$ so that
\[ \bigoplus_{i=1}^k \FA p_i \cap I_{d-1} = (0).\]

If $q \in I$ is any other degree $d$ polynomial,
then by maximality its leading polynomial $q'$ cannot be linearly independent from the set $\{p'_1, \ldots, p'_k\}$.
Therefore there exist $\alpha_1, \ldots, \alpha_k \in F$ (i.e. scalars) such that
\[q' = \alpha_1 p'_1 + \ldots + \alpha_k p'_k. \]
This implies that the polynomial
\[q - \sum_{i=1}^k \alpha_i p_i \in I \]
is either $0$ or of degree less than $d$.
This implies that the set $I_d$ is equal to
\[I_d = \bigoplus_{i=1}^k F p_i \oplus I_{d-1}. \]

It now suffices to show that $\FA I_{d-1} \subseteq \bigoplus_{i=1}^k \FA p_i \oplus I_{d-1}$.


Proceed by induction on degree 
to show that for any monomial $m$
one has $m I_{d-1} \subseteq \bigoplus_{i=1}^k \FA p_i \oplus I_{d-1}$.
If $\deg(m) =0$, then the result is trivial.  Next, suppose the result holds for $\deg(m) \leq n$.
Let $m = m_1m_2$, where $\deg(m_2) = 1$.  By the above discussion, $m_2 I_{d-1} \subseteq I_d = \bigoplus_{i=1}^k \FA p_i \oplus I_{d-1}$.
By induction, since $\deg(m_1) < \deg(m)$, $m_1 m_2 I_{d-1}
\subseteq \bigoplus_{i=1}^k m_1 \FA p_i \oplus m_1 I_{d-1} \subseteq \bigoplus_{i=1}^k \FA p_i \oplus I_{d-1}$.

 For item \eqref{it:Adl2}, let $q \in I$ be a degree $D$ polynomial.  By the first part,
\[q = \sum_{i=1}^k q_ip_i + u, \]
where $q_1,\ldots,q_k \in \FA$ and $u \in I_{d-1}$. Since $D \ge d$, at least one $q_i$ is nonzero.
Therefore, by Lemma \ref{eq:LinIndOfHomNcPolys2},
$q' = \sum_{\deg (q_i)=e} q_i' p_i'\in \sum_{i=1}^k \FA_{D-d}^H p'_i$
with $e=\max_i \{\deg(q_i)\}$. The converse is clear.
\end{proof}




Item \eqref{it:Adl2} of Lemma \ref{lem:Adl} says that
for every left ideal $I$ of $\FA$ generated by elements of degree 
 at most $d$ and every $D \ge d$ we have
\begin{equation}
\label{idlformula}
I_D^\ell=\FA_{D-d} I_d^\ell.
\end{equation}

\begin{lemma}
\label{lem:decomp}Let $I \subseteq \FA$ be a left ideal
generated by polynomials of degree at most $d$.
Consider any decomposition of $\FA^H_d$ of the form
\[\FA^H_d = I^{\ell}_d \oplus G,\]
where $G \subset \FA_d^H$.
Then
\[I \cap \FA  G = \FA I^{\ell}_d \cap \FA G = \{0\}.\]
\end{lemma}


\begin{proof}
Suppose $p \in I \cap \FA  G$.
By assertion (1) of Lemma \ref{lem:Adl}, there exist
$p_1, \ldots, p_k \in I$ of degree $d$  
such that the set of leading terms, $p'_1, \ldots, p'_k$,
is independent and there exist 
$q_1,\ldots,q_k \in \FA$, $u \in I_{d-1}$
such that $p= \sum_{i=1}^k q_i p_i + u$.
There also exist a linearly independent set
$v_1,\ldots,v_l \in G$ and polynomials  $s_1,\ldots,s_l \in \FA$
such that $p=\sum_{j=1}^k s_j v_j$. 
Because $I_d^\ell\cap G=(0)$,
the set  $p_1',\ldots,p_k',v_1,\ldots,v_l$
is linearly independent. Further,
$0=p-p=\sum_{i=1}^k q_i p_i+\sum_{j=1} (-s_j) v_j+u.$
By Lemma \ref{eq:LinIndOfHomNcPolys2}
it follows that each $q_i$ and each $s_j$ is $0$.
Hence $p=\sum s_jv_j=0.$

The second equality follows from Example \ref{ABdef}.
\end{proof}

\begin{lemma}
\label{lem:IplusIstarLeading}
Let $I \subseteq \FA$ be a left ideal
generated by polynomials
$p_1, \ldots, p_k \in \FA$ with $\deg p_i \leq d$ for all $i$.
Suppose  $G$ is a subspace of $\FA_d^H$ such that
\[\FA_d^H = I_d^{\ell} \oplus G. \]
If $D \geq d,$ then  the space $\left(I + I^*\right)_{2D}^{\ell}$ is equal to
\begin{equation}
\label{eq:IplusIstarLeading}
(I + I^*)_{2D}^{\ell} = \left[(I_d^{\ell})^* \FA_{2(D-d)}^H I_d^{\ell}\right]
\oplus
\left[G^* \FA_{2(D-d)}^H I_d^{\ell}\right]
\oplus
\left[(I_d^{\ell})^* \FA_{2(D-d)}^H G\right].
 \end{equation}

Consequently, the space
\[
W:= 
G^*
 \FA_{2(D-d)}^H G
\]
satisfies
\[
\FA_{2D}^H = (I + I^*)_{2D}^{\ell} \oplus W.
\]
\end{lemma}

\begin{proof}Each element of $I + I^*$ is of the form
$p + q^*$, where $p, q \in I$.
The leading polynomial of $p$ is in $I_{\deg(p)}^{\ell}$ and the leading polynomial of $q^\ast$ is in
$(I_{\deg(q)}^{\ell})^*$.  We consider two cases.

First, suppose $2D = \deg(p + q^*) < \max\{\deg(p), \deg(q)\}.$
This can only happen when the leading polynomials of $p$ and $q^*$ cancel each other out, that is,
if the leading polynomials of $p$ and $-q^*$ are the same.
Let $\deg(p) = \deg(q) = D'$.
Decompose the space $\FA_{D'}^H$
as
\begin{align}
\label{eq:decompos}
\FA_{D'}^H
& = \FA_{D'-d}^H I_d^{\ell} \oplus \FA_{D'-d}^H G\\
\notag & = \left[(I_d^{\ell} \oplus G)^*\FA_{D'-2d}^H I_d^{\ell}\right] \oplus \left[(I_d^{\ell} \oplus G)^*\FA_{D'-2d}^H G\right]\\
\notag & = \left[(I_d^{\ell})^* \FA_{D'-2d}^H I_d^{\ell}\right] \oplus \left[H^*\FA_{D'-2d}^H I_d^{\ell}\right] \\
\notag &\oplus
\left[(I_d^{\ell})^* \FA_{D'-2d}^H G \right] \oplus \left[H^*\FA_{D'-2d}^H G\right].
\end{align}
Using equations \eqref{idlformula} and \eqref{eq:decompos} respectively, decompose $I_{D'}^{\ell}$ as
\[ I_{D'}^{\ell} = \FA_{D' - d}^H I_d^{\ell} = (I_d^{\ell})^* \FA_{D' - 2d}^H I_d^{\ell} \oplus G^*\FA_{D' - 2d}^H I_d^{\ell},\]
and decompose $I_{D'}^{\ell}$ as
\[ (I^*)_{D'}^{\ell} = (I_{D'}^{\ell})^* = (I_d^{\ell})^* \FA_{D' - 2d}^H I_d^{\ell} \oplus (I_d^{\ell})^*\FA_{D' - 2d}^H G.\]
The leading polynomial of $p$ and $-q^*$ must therefore be in the space
\begin{equation*}
 I_{D'}^{\ell} \cap (I^*)_{D'}^{\ell} =
  (I_d^{\ell})^* \FA_{D' - 2d}^H I_d^{\ell}.
\end{equation*}
Let the leading polynomial of $p$ and $-q^*$ be equal to
\begin{equation}
\label{eq:pprimeisabc}
 p' = -(q')^* = \sum_{i=1}^n (a_i')^* b_i c_i' \in (I_d^{\ell})^* \FA_{D' - 2d}^H I_d^{\ell}
\end{equation}
where each $a_i'$ is the leading polynomial of some $a_i \in I_d$, each $c_i'$ is the leading polynomial of some $c_i \in I_d$, and $b_i \in \FA_{D' - 2d}^H$.
Then
\[p + q^* = \left(p -  \sum_{i=1}^n (a_i)^* b_i c_i \right)
+ \left(q + \sum_{i=1}^n (c_i)^* (b_i)^* a_i  \right)^*,\]
which is a sum of something from $I$ and something from $I^*$, each of degree less than $D'$.  Proceed inductively to reduce $p + q^*$ to a sum of polynomials of degree bounded by $2D$.

Now consider the case where $\deg(p), \deg(q) \leq 2D$.
By hypothesis, $\deg(p + q) = 2D$, so at least one of $p$ or $q$ must be degree $2D$.
If $\deg(p)< 2D$, then $\deg(q) = 2D$ and the leading polynomial of $p + q^*$ is the leading polynomial of $q^*$,
which, by Lemma \ref{lem:Adl}, is an element of
$$(I_d^{\ell})^* \FA_{2(D-d)} I_d^{\ell} \oplus (I_d^{\ell})^* \FA_{2(D-d)} G.$$
If $\deg(q) < 2D$, then $\deg(p) = 2D$ and the leading polynomial of $p + q^*$ is the leading polynomial of $p$,
which, by Lemma \ref{lem:Adl}, is an element of $$(I_d^{\ell})^* \FA_{2(D-d)} I_d^{\ell} \oplus G^* \FA_{2(D-d)} I_d^{\ell}.$$
If $\deg(p) = \deg(q) = 2D$, then the leading polynomial of $p + q^*$ must be the sum of the leading polynomials of $p$ and $q^*$
(which, by assumption, must be nonzero). This is in the space
\[\left[(I_d^{\ell})^* \FA_{2(D-d)} I_d^{\ell} \oplus (I_d^{\ell})^* \FA_{2(D-d)} G\right] +
\left[ (I_d^{\ell})^* \FA_{2(D-d)} I_d^{\ell} \oplus G^* \FA_{2(D-d)} I_d^{\ell}\right]\]
\[ = (I_d^{\ell})^* \FA_{2(D-d)} I_d^{\ell} \oplus
H^* \FA_{2(D-d)} I_d^{\ell}
\oplus
(I_d^{\ell})^* \FA_{2(D-d)} G.\]
In all cases, the leading polynomial of an element of $I + I^*$ is in the space
(\ref{eq:IplusIstarLeading}).
\end{proof}

\begin{prop}
\label{cor:computeIplusIstar}
Let $I \subseteq \FA$ be a left ideal
generated by polynomials 
with degree bounded by $d$.
\begin{enumerate}
 \item The space $(I + I^*)_{2d-1}$ is equal to
\[ (I + I^*)_{2d-1} = I_{2d-1} + I^*_{2d-1}.\]
\item  Choose, by Lemma \ref{lem:Adl}
polynomials $p_1, \ldots, p_k$ so that
\[I = \sum_{i=1}^k \FA p_i + I_{d-1}. \]
If  $\{q_1, \ldots, q_{\ell}\}$ is a basis for
$I_{d-1},$  then the set
\begin{equation}
\label{eq:spanOfIplusIstarHerm}
 \{ mp_i + p_i^*m^*:\ m\ \text{monomial},\ \deg(mp_i) <2d\} \cup \{q_1 + q_1^*, \ldots, q_{\ell} + q_{\ell}^* \}
\end{equation}
spans $(I + I^*)_{2d-1} \cap \FA_h$.
\end{enumerate}

\end{prop}

\begin{proof}
Let $p, q \in I$
with $\deg(p) \geq 2d$ and
$\deg(p + q^*) < 2d$.
This can only happen if 
$\deg(p) = \deg(q)$
and the leading polynomials
$p'$ and $q'$ of $p$ and $q$
respectively satisfy $(p')^* = -q'$.
As in (\ref{eq:pprimeisabc},
we see that 
\[ p' = -(q')^* = \sum_{i=1}^n (a_i')^*b_ic_i',\]
where $a_i', c_i'$ are the leading polynomials
of some $a_i, c_i \in I$.
Therefore
\[ p + q^* = \left(p - \sum_{i=1}^n (a_i)^*b_ic_i\right) + \left(q + \sum_{i=1}^n (c_i)^*b_i^*a_i\right)^*, \]
which is a sum of an element of $I$ of degree less than $\deg(p)$
and an element of $I^*$ of degree less than $\deg(p)$.
We proceed inductively to show that $p + q^* \in I_{2d-1} + I_{2d-1}^*$.

Further, by Lemma \ref{lem:Adl} $I_{2d-1}$ is spanned by polynomials
of the form $mp_i$, with $m$ a monomial and $\deg(mp_i) < 2d$, together with the $q_j$.
A symmetric polynomial $p \in I_{2d-1} + I_{2d-1}^*$ is therefore equal to
\begin{equation}
\label{eq:psymmetric}
  p = \sum_{\deg(mp_i) < 2d} A_{mp_i} mp_i + \sum_{\deg(np_j) < 2d} B_{np_j} p_j^*n^* + \sum_{n=1}^{\ell} C_n q_n + \sum_{r=1}^{\ell} D_n q_n^*
\end{equation}
for some sufficiently defined $A_{mp_i}, B_{np_j}, C_n \in F$.
But $p$ being symmetric means $p = \frac{1}{2} (p + p^*)$, so
\[ p = \frac{1}{2} \left [\sum A_{mp_i} (mp_i + p_i^*m^*) + 
\sum B_{np_j} (np_j + p_j^*n^*) + \sum (C_n + D_n)(q_n + q_n^*)\right]. \]
Therefore (\ref{eq:spanOfIplusIstarHerm}) is a spanning set for
$(I + I^*)_{2d-1} \cap \FA_h$.
\end{proof}

\begin{lemma}
\label{lem:basis4W2DH}
 Let $G$ and $W$ be as in
Lemma \ref{lem:IplusIstarLeading}.
Let $q_1, \ldots, q_k$ be a basis
for $\FA_{D-d}^H G$.  Then
the set of products $q_i^*q_j$, where $1 \leq i,j \leq k$,
is  a basis
for $W$.
\end{lemma}

\begin{proof}
Given that $q_1, \ldots, q_k$ are a basis
for $\FA_{D-d}^H G$,
then 
\[\FA_{D-d}^H G = \sum_{i=1}^k F q_i \]
Using Lemma \ref{lem:IplusIstarLeading},
\[W =  \left(\sum_{i=1}^k F q_i\right)^*\left(\sum_{i=1}^k F q_i\right)
 = \sum_{i=1}^k\sum_{j=1}^k F q_i^*q_j.\]
Therefore the $q_i^*q_j$ span $W_{2D}^H$.
Further, by Lemma \ref{eq:LinIndOfHomNcPolys}, the $q_i^*q_j$ must be linearly
independent.
\end{proof}

The reader who is only interested in the 
proof of Theorem \ref{thm:main}
can skip from here to the next section.

\begin{prop}
\label{prop:lowdeg}
If $I \subseteq \FA$ is a left ideal
generated by polynomials of degree at most $d$,
then there is a subspace $V$ of $\FA$ such that 
\begin{enumerate}
\item 
 \label{it:lowdeg1}
    \[\FA = I \oplus V, \]
    and 
    \[ V = \FA V_d^H \oplus V_{d-1}.\]
In particular, 
\[ \FA_d^H = I_d^{\ell} \oplus V_d^H \quad \text{and} \quad \FA_{e} = I_{e} \oplus V_{e}\quad \forall e \geq d-1; \mbox{ and}\]
\item \label{it:lowdeg2} 
 if  $\sum_{j=1}^{\ell} q_j^*q_j \in I + I^*,$ then $q_i \in I \oplus V_{d-1}$ for each $j$.
\end{enumerate}
\end{prop}

\begin{proof}
To prove item \eqref{it:lowdeg1}, first choose a space $G \subset \FA_d^H$
so that
\begin{equation}
 \label{eq:IplusG}
   \FA_d^H = I_d^{\ell} \oplus G.
\end{equation}

By Lemma \ref{lem:decomp}, 
$I \cap \FA G = (0)$.
Next, choose $U \subset \FA_{d-1}$ so that
$\FA_{d-1} = I_{d-1} \oplus U$. In particular $U \cap I = (0)$.

 Of course $U\cap \FA G =\emptyset$
  because $U\subset \FA_{d-1}$ and $\FA G \subset \bigoplus_{i=d}^{\infty} \FA_i^H$.
 Given a word $m$, if the degree of $m$ is $d$ or less, then evidently
 $m\in I_{d-1} \oplus U \subset I\oplus \FA G \oplus U$.  If the degree of $m$ exceeds $d$,
 then $m=pw$ where $w$ is a word of length $d$ and $p$ is a word.
By equation \eqref{eq:IplusG}, $w=h'+g$
for some $h' \in I_d^{\ell}$ and some $g \in G$.
Let $h \in I_d$ be such that the leading polynomial
of $h$ is $h'$ so that $h'-h \in \FA_{d-1}$.
 Thus, $ph \in I$ and $pg \in \FA G$ and it follows that
  $m = ph + pg + p(h' - h) \in I\oplus \FA G \oplus \FA_{d-1}$. 
Consequently,
\[
  \FA = I\oplus \FA G \oplus U. 
\]

Let $V$ be equal to
\[ V = \FA G + U.\]
The space $\FA G$ has polynomials whose terms have degree
at least $d$, whereas the space $U$ has polynomials of degree less than $d$.
Therefore $U = V_{d-1}$.
Further, this implies that $V_{d}^H$ must
be contained in $\FA G$.  The homogeneous degree
$d$ polynomials in $\FA G$ are precisely those in $G$.
Therefore $G = V_d^H$.

Turning to item \eqref{it:lowdeg2},
suppose there exists a sum of squares
$\sum_{j=1}^{\ell} q_j^*q_j \in I + I^*$.
Decompose each $q_j$ as
\[q_j = q_{j, I} + q_{j, \FA  V^H_d} + q_{j, V_{d-1}} \]
where $q_{j, W} \in W$ for each space $W$ used.
This implies
\begin{align}
\notag \sum_{j=1}^{\ell} q_j^*q_j = \sum_{j=1}^{\ell} (q_{j, I} + q_{j,
\FA  V^H_d} + q_{j, V_{d-1}})^*(q_{j, I}
+ q_{j,\FA  V^H_d} + q_{j, V_{d-1}})\\
 \label{eq:sosintothree}
 = \sum_{i=1}^{\ell} \left[(q_{i, I} + q_{i, \FA
 V^H_d} + q_{i, V_{d-1}})^*q_{i, I} + q_{i, I}^*(q_{i,
 \FA  V^H_d} + q_{i, V_{d-1}})\right]
 \\
 \label{eq:lastline}
 + \sum_{j=1}^{\ell} (q_{j, \FA  V^H_d} + q_{j,
 V_{d-1}})^*(q_{j, \FA  V^H_d} + q_{j, V_{d-1}}) \in I + I^*.
 \end{align}
Since (\ref{eq:sosintothree}) is in $I + I^*$,
this implies that \ref{eq:lastline}
is in $I + I^*$.

Assume
\[\sum_{j=1}^{\ell} (q_{j, \FA  V^H_d})^*(q_{j, \FA  V^H_d})
\neq 0\]
Suppose $\displaystyle\sum_{j=1}^{\ell} (q_{j, \FA  V^H_d})^*(q_{j, \FA  V^H_d})$ is degree $2D$, for $D \geq d$, and let each $q_{j,\FA V_d^H}$ be equal to
\[q_{j,\FA V_d^H} = v_j + w_j,\]
where $v_j \in \FA_{D-d}^H V_d^H$ and where $\deg(w_j) < D$.
Also, by definition each $q_{j,V_{d-1}}$ must have degree less than $d$.
Therefore
\begin{align}
\notag \sum_{j=1}^{\ell} (q_{j, \FA  V^H_d} + q_{j,
 V_{d-1}})^*(q_{j, \FA  V^H_d} + q_{j, V_{d-1}}) = \sum_{j=1}^{\ell} v_j^* v_j\\
\label{eq:deglessthan2D} + \sum_{i=1}^{\ell} \left[(v_i + w_i + q_{i, V_{d-1}})^*(w_i + q_{i, V_{d-1}}) + (w_i + q_{i, V_{d-1}})^* v_{i}\right]
 \end{align}
 We see that (\ref{eq:deglessthan2D}) has degree less than $2D$ and that $$\sum_{j=1}^{\ell}v_j^*v_j \in \FA_{2D}^H.$$
 Therefore the leading polynomial of $(\ref{eq:lastline})$ is $$\sum_{j=1}^{\ell}v_j^*v_j \in (V_d^H)^*\FA_{2(D-d)}^H V_d^H.$$
Since $(\ref{eq:lastline})$ is in the space $I + I^*$, this implies that
 $$\sum_{j=1}^{\ell}v_j^*v_j \in (I + I^*)_{2D}^{\ell}.$$
 By Lemma \ref{lem:IplusIstarLeading} and by the decomposition of $\FA_{D'}^H$ in (\ref{eq:decompos}), this implies that \[
 \sum_{j=1}^{\ell}v_j^*v_j \in (I + I^*)_{2D}^{\ell} \cap (V_d^H)^*\FA_{2(D-d)}^H V_d^H = (0).
 \]
 This implies that each $v_j = 0$, which is a contradiction.
Therefore each $q_{j, V_d^H} = 0$, which implies that each $q_i \in I \oplus V_{d-1}$.
\end{proof}

With these lemmas, we proceed to prove Theorem \ref{thm:lowdeg}.

\subsection{Proof of Theorem \ref{thm:lowdeg} }

\begin{proof}
The direction $(1) \Rightarrow (2)$ follows by definition, and the direction $(2) \Rightarrow (3)$ is clear.

Assume (3).
Decompose $\FA_{d-1}$ as
\[\FA_{d-1} = I_{d-1} \oplus V\]
for some space $V$.
Decompose $\FA_d^H$ as
\[\FA_d^H = I_d^{\ell} \oplus V_d^H\]
for some space $V_d^H \subset \FA_d^H$.    Then as in Proposition \ref{prop:lowdeg},
\[\FA = I \oplus \FA V_d^H \oplus V,\]
where $V$ takes the place of $V_{d-1}$.

Suppose
\[\sum_{j=1}^{k} q_j^*q_j \in I + I^*. \]
By Proposition \ref{prop:lowdeg}, each $q_j \in I \oplus V$. Let each $q_i$ be equal to
\[q_j = \iota_j+ v_j,\]
where $\iota_j \in I$ and $v_j \in V$. Then
\begin{align}
\notag \sum_{j=1}^{k} q_j^*q_j &=
\sum_{i=1}^{\ell} v_i^*v_i \\
\label{eq:Iline} &+ \sum_{j=1}^{k} [\iota_j^*v_j + v_j^*\iota_j + \iota_j^*\iota_j].
\end{align}
The line (\ref{eq:Iline}) is in $I + I^*$, which implies that $\displaystyle\sum_{i=1}^{k}v_i^*v_i \in I + I^*$. By (3), each $v_i$ must be equal to $0$. Therefore $q_j = \iota_j \in I$ for each $j$. This implies (1).
\end{proof}

\section{An Algorithm for Computing $\rr{I}$}
\label{sec:alg}

   It is of interest  to describe and to
     compute      the {\qrr} of a left ideal $I$, in part
 because of its close relation to the $\Pi$-{\rad} of $I$.
This section gives an algorithm and theory
which shows that the algorithm does indeed have very
desirable  properties.

\subsection{The Real Algorithm}
\label{sub:alg}

The following is an algorithm for computing
$\rr{I}$ given a finitely-generated
left ideal
$I \subset \FA$.
Here, let $I = \sum_{i=1}^k \FA p_i$, where the  $p_i \in \FA$
are polynomials with $\deg p_i \leq d$.

\begin{enumerate}
\item
\label{A1} Let $I^{(0)} = I$.

\item
\label{A2} 
At each step $k$ we have an ideal $I^{(k)} \subset \rr{I}$
generated by polynomials of degree bounded by $d$.
Find a sum of squares
$\sum_{i=1}^n q_i^*q_i \in I^{(k)} + {I^{(k)}}^*$ such that for
each $j$ one has $q_j \not\in I$ and $\deg (q_j) < d$.
If such a sum of squares is not obvious, the
following algorithm, which we will refer to as the \textbf{SOS Algorithm}, either computes such a sum of squares or proves that none exists.

\textbf{SOS Algorithm}

\begin{enumerate}
\item
\label{Aia} Find a complementary space $V^{(k)} \subset \FA_{d-1}$
such that
\[\FA_{d-1} = I_{d-1}^{(k)} \oplus V^{(k)}.\]
Find a basis $\{v_1, \ldots, v_{\ell}\}$ for $V^{(k)}$.

\item
\label{Aib}
 Parameterize the symmetric elements of $I^{(k)} + {I^{(k)}}^*$
 which appear in the span of
$\{v_i^*v_j\}$ as
\[\begin{pmatrix}
v_1 \\ \vdots \\ v_{\ell}
\end{pmatrix}^T
(\alpha_1A_1 + \ldots \alpha_mA_m)
\begin{pmatrix}
v_1 \\ \vdots \\ v_{\ell}
\end{pmatrix},\]
for some Hermitian matrices $A_i \in F^{\ell \times \ell}$.

\begin{itemize}
\item To find the matrices $A_1, \ldots, A_m$, one does the following.

Find a basis $\iota_1, \ldots, \iota_{p}$ for the symmetric elements of
\[\left(I^{(k)}+I^{(k)^*}\right)_{2d-2}.\] 
The set (\ref{eq:spanOfIplusIstarHerm}) in Proposition \ref{cor:computeIplusIstar} gives
a spanning set from which one can choose a maximal linearly independent
subset.  Solve the equation
\begin{equation}
\label{eq:theSofEq}\begin{pmatrix}
v_1 \\ \vdots \\ v_{\ell}
\end{pmatrix}^T
\left(
\begin{array}{ccc}
a_{11}&\ldots& a_{1\ell}\\
\vdots&\ddots& \vdots\\
a_{\ell 1}&\ldots&a_{\ell\ell}
\end{array}
\right)\begin{pmatrix}
v_1 \\ \vdots \\ v_{\ell}
\end{pmatrix} = \alpha_1 \iota_1 + \ldots + \alpha_{p} \iota_p.
\end{equation}

This amounts to solving a system of linear equations
in variables $a_{ij}$ and $\alpha_j$, which system is given by setting the
coefficient of each monomial in (\ref{eq:theSofEq}) equal to zero.
Project this set of solutions onto the coordinates
$a_{ij}$ to get the set
\[\{A = (a_{ij})_{1 \leq i,j \leq \ell} \mid
\exists \alpha_1, \ldots, \alpha_m\ \colon (\ref{eq:theSofEq})\
\mathrm{holds}\}.\] Find a basis $A_1, \ldots, A_m$ for this
new projected space.
\end{itemize}

\item
\label{Aic} Solve the following linear matrix inequality for $(\alpha_1, \ldots, \alpha_m)$.
\[\alpha_1A_1 + \ldots + \alpha_mA_m \succeq 0 \quad \text{ and } \quad (\alpha_1, \ldots, \alpha_m)
 \neq 0.\]

\begin{itemize}
\item If there is a solution $(\alpha'_1, \ldots, \alpha'_m) \neq 0$,
then let $q_1, \ldots, q_n$ be the polynomials
\[\begin{pmatrix}
q_1 \\ \vdots \\ q_n
\end{pmatrix} = \sqrt{\alpha'_1A_1 + \ldots \alpha'_mA_m}\begin{pmatrix}
v_1 \\ \vdots \\ v_{\ell}
\end{pmatrix}.\]
Then $\sum_{i=1}^n q_i^*q_i \in I^{(k)} + {I^{(k)}}^*$ is such that
each $q_j \not\in I$ and $\deg q_j < d$.

\item If this linear matrix inequality has no solution, then there exists no sum of squares
    $\sum_{i=1}^n q_i^*q_i \in I^{(k)} + {I^{(k)}}^*$ such that
each $q_j \not\in I$ and $\deg q_j < d$.
\end{itemize}
\end{enumerate}

\item
\label{A3}
 If there exists a sum of squares $\sum_{i=1}^n q_i^*q_i
 \in I^{(k)} + {I^{(k)}}^*$ such that
each $q_j \not\in I$ and $\deg q_i < d$, then
let $I^{(k+1)} = I^{(k)} + \sum_{i=1}^n \FA q_i$, let $k = k+1$,  note
that $I^{(k+1)}$ is again an ideal,
and go to step 2.

\item
\label{A4}
 If there exists no sum of squares
 $\sum_{i=1}^n q_i^*q_i \in I^{(k)} + {I^{(k)}}^*$ such that
each $q_j \not\in I$ and $\deg q_j < d$, then output $I^{(k)}$ and end
the Algorithm.
\end{enumerate}
\qed

The following theorem presents some appealing properties of the Real Algorithm.

\begin{thm}
\label{thm:algorStops}
Let $I$ be the left ideal generated by polynomials
$p_1, \ldots, p_k$, with $\deg(p_i) \leq d$
for each $i$.  The following are true for applying the Algorithm described in
\S \ref{sub:alg} to $I$. 

\begin{enumerate}
\item
 \label{degreed}
   This Algorithm involves only computations of polynomials which have degree less than $d$.
\item The Algorithm is guaranteed to terminate in
a finite number of steps.

\item When the Real Algorithm terminates, it outputs the ideal $\rr{I}$.
\end{enumerate}
\end{thm}

\begin{proof}\begin{enumerate}
\item This is clear from the steps of the Algorithm.

\item In the Algorithm, at each step
the ideal $I^{(k+1)} = I^{(k)} + \sum_{i=1}^n \FA q_i$
is formed from some polynomials
$q_i$ with degree bounded by $d -1$.
The chain $I^{(k)}_{d-1}$ is strictly increasing and hence, in view of item
 \ref{degreed},
\[
   I^{(0)}_{d-1} \subsetneq I^{(1)}_{d-1} \subsetneq I^{(2)}_{d-1} \subsetneq \cdots .
\]
  Since each $I^{(k)}_{d-1}$ is a subset of the finite dimensional vector
  space $\FA_{d-1}$, this chain, and thus the Algorithm, terminates.

\item First of all, $I^{(0)} \subset \rr{I}$.
Suppose by induction that $I^{(k)} \subset \rr{I}$.
If there exists a  sum of squares $\sum_{i=1}^n q_i^*q_i \in I^{(k)}$ such that $q_i \not\in I$ for each $i$,
it follows that
\[\sum_{i=1}^n q_i^*q_i \in I^{(k)} \subset \rr{I}.\]
This implies that $q_i \in \rr{I}$ for each $i$. Therefore
\[I^{(k)} + \sum_{i=1}^n \FA q_i \subseteq \rr{I}.\]
Continue this process until there is an $I^{(k')} \subset
\rr{I}$ such that there exists no such sum of squares.
By Theorem \ref{thm:lowdeg}, the left ideal $I^{(k')}$ is real, and hence equal to $\rr{I}$. The algorithm also stops at this point, and so $\rr{I}$ is the output.
\end{enumerate}
\end{proof}

\subsection{An Example of Applying the Algorithm}

We apply the Algorithm on the left ideal
\[I = \FA \left(\left[x_1^*x_1 + x_2x_3x_3^*x_2^*\right]^*\left[x_1^*x_1
+ x_2x_3x_3^*x_2^*\right] + x_4^*x_4\right).\]
We see that
\[ p:= [x_1^*x_1 + x_2x_3x_3^*x_2^*]^*[x_1^*x_1 + x_2x_3x_3^*x_2^*]
+ x_4^*x_4\]
is in $I$ and is a sum of squares.
We take $q_1 = x_1^*x_1 + x_2x_3x_3^*x_2^*$ and $q_2 = x_4$, which have degree less than $8$, to form the ideal $I^{(1)}$ equal to
\[I^{(1)} = \FA (x_1^*x_1 + x_2x_3x_3^*x_2^*) + \FA x_4.\]
Note $I^{(0)} \subset I^{(1)}$.

In $I^{(1)}$ there is a sum of squares
\[x_1^*x_1 + x_2x_3x_3^*x_2^* \in I^{(1)}.\]
The ideal
 $I^{(2)}$ is constructed similarly and is
\[I^{(2)} = \FA x_1 + \FA x_3^*x_2^* + \FA x_4.\]
At this point it may not be obvious that whether or not there is a nontrivial sum of squares
in $I^{(2)} +I^{(2)^*}$. We turn to the SOS Algorithm to
either find such a sum of squares or prove that one does not exist.

Since
$I^{(2)}$ is generated by polynomials of degree bounded by two,
let $d = 2$.

\textit{Step  \ref{Aia} }.
First we find a complementary space $V^{(2)}$.
The space $I^{(2)}_1$
is the span
\[I^{(2)}_1 = \mathrm{span}\{x_1, x_4 \}. \]
Choose $V^{(2)}$ to be
\[V^{(2)} = \mathrm{span}\{x_1^*, x_2, x_2^*, x_3, x_3^*, x_4^*, 1\}\]
so that $\FA_1 = I^{(2)}_1 \oplus V^{(2)}$.

\textit{Step \ref{Aib}}.
Elements of $I^{(2)} + I^{(2)^*}$ are sums of monomials with the rightmost letters being
$x_1, x_3^*x_2^*$ or $x_4$, or the leftmost letters being $x_1^*, x_2x_3$ or $x_4^*$. Because $x_1, x_4 \not\in V^{(2)}$, the only such polynomials in the span of the
$v_i^*v_j$ are polynomials of the form
$\alpha x_3^*x_2^* + \beta x_2 x_3$, where $\alpha, \beta \in F$.
Consequently, the only symmetric elements of $I^{(2)} + {I^{(2)}}^*$
in $\mathrm{span}\{v_i^*v_j\}$
are polynomials of the form $\alpha (x_3^*x_2^* + x_2x_3)$, with $\alpha \in F$.

\textit{Step \ref{Aic}}. \
We then parameterize all elements of
$\left(I^{(2)} + {I^{(2)}}^*\right) \cap\ \mathrm{span}\{v_i^*v_j\}$
as
\[\alpha \begin{pmatrix}
x_1^* \\ x_2 \\ x_2^* \\ x_3 \\ x_3^* \\ x_4^* \\ 1
\end{pmatrix}^*
\left(
\begin{array}{ccccccc}
0&0&0&0&0&0&0\\
0&0&0&0&0&0&0\\
0&0&0&1&0&0&0\\
0&0&1&0&0&0&0\\
0&0&0&0&0&0&0\\
0&0&0&0&0&0&0\\
0&0&0&0&0&0&0
\end{array}
\right)
\begin{pmatrix}
x_1^* \\ x_2 \\ x_2^* \\ x_3 \\ x_3^* \\ x_4^* \\ 1
\end{pmatrix} \]

The linear matrix inequality
\[\alpha
\left(
\begin{array}{ccccccc}
0&0&0&0&0&0&0\\
0&0&0&0&0&0&0\\
0&0&0&1&0&0&0\\
0&0&1&0&0&0&0\\
0&0&0&0&0&0&0\\
0&0&0&0&0&0&0\\
0&0&0&0&0&0&0
\end{array}
\right)
\succeq 0 \]
has no nonzero solution in $\alpha$ since the matrix
in question is neither positive semi-definite
nor negative semi-definite.
This means we go to Step \ref{A4} of the Algorithm which says stop.
Therefore
\[\rr{I} = \FA x_1 + \FA x_3^*x_2^* + \FA x_4.\]
\qed

\section{A Nullstellensatz for $\FNSx$}
\label{sec:pfmain}

We provide the remaining ingredients for the proof of Theorem \ref{thm:main}, 
namely Theorem \ref{thm:existsLfinite} and Proposition \ref{prop:finite}. 
The proof also depends on the Real Algorithm.

\subsection{Existence of Positive Linear Functionals}

The following is the main technical result used in the proof of Theorem \ref{thm:main}.

\begin{thm}
\label{thm:existsLfinite}
 Let $I$ be a finitely-generated real left ideal.
Then there exists a positive hermitian $F$-linear functional $L$ on $\FA$ such that 
$$I=\{a \in \FA \mid L(a^*a) =0\}.$$
\end{thm}

\begin{proof}
Let $I$ be generated by a set of polynomials with degree bounded by $d$. 
We will first construct a linear functional $L$ on $\FA_{2d-2}$ such that
\begin{enumerate}
 \renewcommand{\labelenumi}{(\roman{enumi})}
\item $L((I+I^\ast) \cap \FA_{2d-2})=0$,
\item $L(a^\ast a) > 0$ for every $a \in \FA_{d-1} \setminus I$ and
\item $L(a^*) = L(a)^*$ for every $a \in \FA_{2d-2}$.
\end{enumerate}

Choose, by Proposition \ref{prop:lowdeg}, a subspace $V$
of $\FA$ such that
\begin{equation*}
 \FA = I \oplus \FA V_d^H \oplus V_{d-1}.
\end{equation*}
and
\begin{equation}
 \label{eq:fixVd}
\FA_{e}=I_{e}\oplus V_{e}
\end{equation}
 for each $e\ge d-1$. 
Let $q_1, \ldots, q_k$ span $V_{d-1}$,
and let $q = (q_1, \ldots, q_k)$.

Let $M_k(F)_h$ be the set of all hermitian $k \times k$ matrices with entries in $F$.
(If $F = \RR$, then this is the set of symmetric matrices in $M_k(\RR)$.)
The real vector space  $M_k(F)_h$ carries the (real-valued) inner
product $\langle C,D \rangle=\operatorname{Tr}(C D).$ 

Let $B_1, \ldots, B_k$ be an orthonormal basis for the subspace 
 $\{B \in M_k(F)_h \mid q^* B q \in I + I^*\}$
and let
$A_1, \ldots, A_m$ be its completion to an orthonormal basis for $M_k(F)_h$.
Consider $M(\alpha, \beta)$ defined by
\[ M(\alpha, \beta) = \sum_{i=1}^m \alpha_i A_i + \sum_{j=1}^k \beta_jB_j, \quad \alpha \in \RR^m, \beta \in \RR^k.\]
Since the $A_i$ and $B_j$ form a basis for $M_k(F)_h$,
the function $M(\alpha, \beta) \colon \RR^m \times \RR^k \to M_k(F)_h$ is onto.
Therefore the set $\mathcal{C}$ defined by
\[\mathcal{C} = \{ \beta \in \RR^k \mid \exists \alpha \colon M(\alpha, \beta) \succ 0\}\]
is a nonempty convex set.

If $0 \not\in \mathcal{C}$, then
there exists $x \neq 0$ such that
\[ \mathcal{C} \subset \{y \mid \langle x, y \rangle \geq 0 \}.\]
Let $B = \sum_{j=1}^k x_jB_j$.
Then
for each positive definite matrix in $M_k(F)_h$
which, since $M$ is onto, must be of
 the form $M(\alpha, \beta)$ for some $\alpha, \beta$,
\[ \langle M(\alpha, \beta), B \rangle = \langle x, \beta \rangle \geq 0.\]
Therefore the matrix $B \succeq 0$.
This is a contradiction since $I$ is real,
but
 $q^*Bq$
 is a sum of squares in $I + I^*$
of elements which are not in $I$.
Therefore, $0 \in \mathcal{C}$, which implies that there exists
$A = \sum_{i=1}^m \alpha_i A_i \succ 0$.

This $A$ is the key to the construction of $L$.
Note that $\langle A, B \rangle = 0$ for every $B \in M_k(F)_h$ 
such that $q^* B q \in I + I^*$. To show that if fact
 $\operatorname{Tr}(AB) =0$ whenever 
  $B \in M_k(F)$ and $q^*Bq \in I + I^*$, we consider
 two cases depending on the base field $F$. 
If $F = \RR$, then 
 $q^*(B + B^*)q \in I + I^*$ so
\[  2\operatorname{Tr}(AB)=\langle A, B + B^*\rangle = 0. \]
If $F = \CC$, then $q^*(B + B^*)q$ and $q^*(iB - iB^*)q$
are both in  $ I + I^*$ so
\[2\operatorname{Tr}(AB) = \langle A, B 
+ B^*\rangle -i \langle A, iB - iB^* \rangle = 0.\]

Next, note that, using equation \eqref{eq:fixVd},
\[ \FA_{2d-2} = \FA_{d-1}^*\FA_{d-1} = I_{d-1}^*I_{d-1} + I_{d-1}^*V_{d-1} + V_{d-1}^*I_{d-1} + V_{d-1}^*V_{d-1}.\]
Therefore each $p \in \FA_{2d-1}$ can be expressed as $p = \iota + q^*Bq$, where $\iota \in I_{2d-2} + I_{2d-2}^*$
and $B \in M_k(F)$. Define $L$ on $\FA_{2d-2}$ to be 
$$L(p)=  L(\iota + q^*Bq) = \operatorname{Tr}(AB).$$
In particular,  $L((I+I^*)_{2d-2}) = \{0\}$.
If $p$ can also be expressed as $p = \tilde{\iota} + q^*\tilde{B}q$, with $\tilde{\iota} \in I_{2d-2} + I_{2d-2}^*$
and $\tilde{B} \in M_k(F)$, then $\tilde{\iota} - \iota = q^*(B - \tilde{B})q \in I + I^*$.
By the previous paragraph, $\operatorname{Tr}(A(B - \tilde{B})) = 0$,
which implies that $L$ is well-defined.
 Also, we see
\[ L([\iota + q^*Bq]^*) = \operatorname{Tr}(AB^*) = \operatorname{Tr}(AB)^* = L(\iota + q^*Bq)^*.\]
Finally, if $a \in \FA_{d-1} \setminus I_{d-1}$, then
an application of equation \eqref{eq:fixVd}
 shows $a = a_I + \alpha^* q$, for some
 $a_I \in I_{d-1}$, and $0 \neq \alpha \in F^k$.
Since $A \succ 0$, 
\[ L(a^*a) = L(q^*\alpha\alpha^*q) = \operatorname{Tr}(A\alpha\alpha^*) = \alpha^*A\alpha > 0.\]

Next we extend $L$ inductively by degree.
Suppose that $L$ is defined on $\FA_{2D-2}$,
$D \geq d$, and it satisfies properties (i)-(iii) with $d$ replaced by $D$.
We set about to extend $L$ to $\FA_{2D}$.
The extension will satisfy properties (i)-(iii) with $d$ replaced by $D+1$.

First we address degree $2D - 1$.
Write the disjoint decomposition of the space where
we must define our extended $L$ as
\[
\FA_{2D-1}^H = (I + I^*)_{2D-1}^{\ell} \oplus W_{2D-1}^H.
\]
for some subspace $W_{2D-1}^H$.
We define $L$ to be 0 on $W_{2D-1}^H$
and turn to defining $L$ on  $(I + I^*)_{2D-1}^{\ell}$
so as to meet the key constraint 
$L((I + I^*)_{2D-1}) = \{0\}$.

Let $p'$ be in  $(I+I^*)^{\ell}_{2D-1}$,
and let $p \in (I + I^*)_{2D-1}$ be such that
$p'$ is the leading polynomial of $p$.
Define $L(p')$ to be $L(p'-p)$.
To prove that $L(p')$ is well-defined suppose that $p'$ is also the leading polynomial of
some $\tilde{p} \in (I + I^*)_{2D-1}$. The polynomial $p-\tilde{p}$ clearly belongs to
$(I + I^*)_{2D-2}$, hence $L(p-\tilde{p})=0$ by assumption. It follows that $L(p'-p)=L(p'-\tilde{p})$.
The definition of $L(p')$ implies that $L(p)=L(p')+L(p-p')=0$ for every $p \in (I + I^*)_{2D-1}$.
Also note that
\[ L[(p')^*] = L[(p')^* - p^*] = L[p' - p]^* = L[p']^*.\]

Next we extend $L$ to degree $2D$.
As in the degree $2D-1$ case, $L$ can be 
extended to $(I + I^*)_{2D}^{\ell}$ to make $L((I + I^*)_{2D}) = \{0\}$.

By Lemma \ref{lem:IplusIstarLeading},
\[\FA_{2D}^H = (I + I^*)_{2D}^{\ell} \oplus W_{2D}^H\]
where
\[W_{2D}^H:= (V_d^{H})^* \FA_{2(D-d)}^H V_d^{H}\]
It follows from Lemma \ref{lem:decomp} that
\[ \FA_D^H = I_D^{\ell} \oplus V_D^H.\]
Note $V_D^H = \FA_{D-d}^H V_d^H$.
Let 
$r_1, \ldots, r_k$ be a basis for
$V_{D}^H$.
By Lemma \ref{lem:basis4W2DH}, the
set of products $r_i^*r_j$
is a basis for $W_{2D}^H$.
For these basis elements,
define $L$ to be
$L(r_i^*r_i) = c$,
 where $c > 0$ is yet to be determined,
and $L(r_i^*r_j) = 0$ for $i \neq j$.
Clearly, $L(a^\ast)=L(a)^\ast$ for every
$a \in \FA_{2D}$.

By Proposition \ref{prop:lowdeg},
\[\FA_D = I_D \oplus V_D \quad \text{and} \quad V_D = V_D^H \oplus V_{D-1}.\]
Let $r_{k+1}, \ldots, r_n$ be a basis for $V_{D-1}$ so that
$r_1, \ldots, r_n$ is a basis for $V_D$.  Let $r = (r_1, \ldots, r_k)$
and $\bar{r} = (r_{k+1}, \ldots, r_n)$.
If $a \in \FA_D \setminus I_D$, then $a$ is of the form
$a = \iota + \alpha^* r + \bar{\alpha}^* \bar{r}$ for some $\iota \in I$, $\alpha \in F^k$,
$\bar{\alpha} \in F^{n-k}$, and at least one of $\alpha$ and  $\bar{\alpha}$ is nonzero.
We see
\[
L(a^* a)=\left[\begin{array}{cc} \alpha^* &\bar{\alpha}^* \end{array} \right]
\left[\begin{array}{cc} c I_k & R \\ R^* & S \end{array} \right]
\left[\begin{array}{c} \alpha \\ \bar{\alpha} \end{array} \right],
\]
where
the $ij^{th}$ entry of $S$ is $L(r_{k+i}^*r_{k+j})$ and the $ij^{th}$ entry
of $R$ is $L(r_{k+i}^*r_{j})$.
Therefore $L(a^*a) > 0$ for all $a \in \FA_D \setminus I_D$ 
if and only if the matrix $\left[\begin{array}{cc} c I_k & R \\ R^* & S \end{array} \right]$
is positive definite.
Note that if $\bar{\alpha} \neq 0$, then from an induction hypothesis,
\[
\bar{\alpha}^* S \bar{\alpha} = L((\bar{\alpha}^*\bar{r})^*(\bar{\alpha}^* \bar{r}) ) > 0
\]
since $\bar{\alpha}^* \bar{r} \in \FA_{D-1} \setminus I_{D-1}$.
Therefore $S \succ 0$.
The last step in defining $L$ therefore is to 
pick $c$ sufficiently large such that the matrix 
$\left[\begin{array}{cc} c I_k & R \\ R^* & S \end{array} \right]$ 
is positive definite.
\end{proof}

\subsection{Relation between $\sqrt[\Pi]{I}$ and $\sqrt[\cR]{I}$}

We will show that $\sqrt[\Pi]{I}=\sqrt[\cR]{I}$
for finitely generated left ideals $I$ in $\FNSx.$ 

\begin{prop}
\label{prop:finite}
If $p_1, \ldots, p_k\in \FNSx$
and $I=\sum_{i=1}^k \FNSx p_i,$ then
\[
\sqrt[\Pi]{I}=\sqrt[\cR]{I}
\]
In particular, suppose $q \in \FNSx$ is such that for each $\Pi$-point
$(X^\prime,v^\prime)$ such that
\[p_1(X^\prime)[v^\prime] = p_2(X^\prime)[v^\prime] = \ldots = p_k(X^\prime)[v^\prime] = 0 \]
that $q(X^\prime)[v^\prime] = 0$. Then for each $\cR$-point
 $(X,v)$ such that
\[p_1(X)[v] = p_2(X)[v] = \ldots = p_k(X)[v] = 0,\]
then $q(X)[v] = 0$ also.
\end{prop}

Recall that $\Pi$-points are, loosely speaking, finite-dimensional representations and
$\cR$-points include infinite-dimensional representations.

\begin{proof}Suppose $q \in \FNSx$, and let $d = \max\{\deg(p_1), \ldots, \deg(p_k), q\}$.
Let $(X,v)$ a representation on some pre-Hilbert space $\mathcal{H}$.
Define $V$ to be the space
\[V = \{p(X)[v]:\ \deg(p) \leq d\} \subset \mathcal{H}. \]
Since the space of polynomials with degree less than or equal to $d$ is finite dimensional, it follows that $V$ is also finite dimensional.
Define $X': V^g \rightarrow V$ by
\[X' = (P_V X_1P_V, \ldots, P_V X_gP_V). \]
Note that $(P_V X_j P_V)^* = P_V X_j^* P_V$.
We claim that for each $r \in \FNSx$ of degree at most $d$,
\begin{equation}
\label{eq:claimconl}
r(X')[v] = r(X)[v].
\end{equation}
Proceed by induction on $\deg(r)$.
If $r$ is a constant, then $r(X')[v] = rv = r(X)[v]$.
Next, consider the case where $r$ is monomial of degree $j \leq d$.
Let $r$ be expressed as
\[r = ym \]
where $y$ is a variable, i.e. $\deg(y) = 1$, and where $m$ is a monomial of degree $j - 1$.
Assume inductively that $m(X')[v] = m(X)[v]$.
Note that $m(X)[v] \in V$ since $\deg(m') \leq d$.
Therefore
\begin{align}
\notag r(X')[v] = y(X')m(X')[v]= P_Vy(X)P_V m(X')[v]=\\
\notag = P_Vy(X)P_V m(X)[v]=P_Vy(X)m(X)[v] = P_Vr(X)[v],
\end{align}
where $y(X)$ denotes evaluating the polynomial $y$ at the $g$-tuple $X$.\index{yX@$y(X)$}
Since $\deg(r) \leq d$, by definition $r(X)[v] \in V$, so $r(X')[v] = r(X)[v]$.
By induction and by linearity, this implies that for any $r \in \FNSx$ with $\deg(r) \leq d$, equation (\ref{eq:claimconl}) holds.

Suppose $q \in \sqrt[\Pi]{I}$. If
\[p_1(X)[v] = p_2(X)[v] = \ldots = p_k(X)[v] = 0,\]
then
\[p_1(X')[v] = p_2(X')[v] = \ldots = p_k(X')[v] = 0.\]
Since $(X',v)$ is a finite-dimensional representation, this implies that
\[q(X)[v] = q(X')[v] = 0.\]
Therefore, $q \in \sqrt[\cR]{I}$.
\end{proof}

\subsection{Proof of Theorem \ref{thm:main}}

\begin{proof}
Let $I$ be a finitely generated left ideal in $\FA=\FNSx$. Then
$$
\sqrt[\mathcal{R}]{I} = \sqrt[\Pi]{I}
$$
by Proposition \ref{prop:finite}. By Theorem \ref{thm:algorStops}, the real left ideal 
$$
J:=  \rr{I}
$$ 
is finitely generated. Then, by Theorem \ref{thm:existsLfinite}, 
there exists a positive hermitian $F$-linear functional $L$ on $\FA$ such that
$$
J=\{a \in \FA¸\mid L(a^\ast a)=0\}.
$$
By the GNS construction,  there exists an $\cR$-point $(\pi,v)$ 
such that $L(a)=\langle \pi(a)v,v \rangle$ for every $a \in \FA$.
(Recall that $V_\pi=\FA/J$ considered as a vector space over $F$ with inner product $\langle p+J,q+J \rangle
=L(q^\ast p)$, $\pi$ is the left regular representation of $\FA$ on $V_\pi$ and $v=1+J$.)
It follows that 
$$J=\J(\{(\pi,v)\}).$$
By the last claim of Lemma \ref{lem:cradicals}, 
$$
\sqrt[\mathcal{R}]{J}=J.
$$
Hence, $\sqrt[\cR]{I}\subseteq J$. By Lemma \ref{lem:qradicals}, also $J \subseteq \sqrt[\cR]{I}$.
\end{proof}

\section{Characterizations of $\sqrt[\cR]{I}$ and $\rr{I}$ in general $\ast$-algebras}
  \label{sec:sqrtRI}
  
The main result of this section is Proposition  \ref{prop:Jaka-procedure} which gives
an iterative procedure for computing $\rr{I}$ in general $\ast$-algebras. We also
discuss the relation of this result to the Real Algorithm.

\subsection{A Topological Characterization of $\sqrt[\cR]{I}$}

Let $\chrisA$ be a $\ast$-algebra. Write $\Sigma_\chrisA$ for the set of all finite sums of elements $a^\ast a$, $a \in \chrisA$.
We assume that $\chrisA_h$ is equipped with the finest locally convex topology, i.e.,
the finest vector space topology 
 whose every neighborhood of zero contains a 
convex balanced absorbing set.
Equivalently, it is
 the coarsest topology 
 for which  every seminorm on $\chrisA_h$ is continuous.
In this case, every linear functional $f$ on $\chrisA_h$ is continuous
since $\vert f \vert$ is a seminorm.

Suppose that $C$ is a convex cone on $\chrisA_h$.
Write $C^\vee$ for the set of all linear functionals $f$
on $\chrisA_h$ such that $f(C) \ge 0$ and write
$C^{\vee \vee}$ for the set of all $v \in \chrisA_h$ such that
$f(v) \ge 0$ for every $f \in C^\vee$.
By the Separation Theorem for convex sets \cite[II.39,
Corollary 5]{d}, $C^{\vee \vee}=\overline{C}$. It follows that for every elements
$a,b \in \chrisA_h$ such that $a+\eps b \in C$ for every real $\eps>0$, we have that
$a \in \overline{C}$.

Note that every $\Sigma_\chrisA$-positive linear functional $f$ on the real vector space $\chrisA_h$ 
extends uniquely to a positive hermitian $F$-linear functional on the $\ast$-algebra $\chrisA$
(namely, take $\tilde{f}(a)=\frac12 f(a+a^\ast)$ if $F=\RR$ and $\tilde{f}(a)=\frac12\left(f(a+a^\ast)-if(ia-ia^\ast)\right)$ if $F=\CC$), 
hence by the GNS construction, see e.g. \cite[Section 8.6]{sch}, there exists a $\ast$-representation $\pi$ of $\chrisA$ and $v \in V_\pi$ such that
$f(a)=\langle \pi(a)v,v \rangle$ for every $a \in \chrisA_h$.

\begin{thm}
\label{nsatz}
Let $I$ be a left ideal in $\ast$-algebra $\chrisA$ and let $\Sigma_I$ be the set of all finite sums
of elements $u^\ast u$ where $u \in I$. Then
\[
\sqrt[\cR]{I}=\{a \in \chrisA \mid -a^\ast a \in \overline{\Sigma_\chrisA-\Sigma_I}\} 
\]
\end{thm}

\begin{proof}
Pick $a \in \chrisA$ and recall that $a \in \sqrt[\cR]{I}$ if and only if
$\pi(a)v=0$ for every $\cR$-point $(\pi,v)$ such that $\pi(x)v=0$
for every $x \in I$. Clearly, the latter is true if and only if
$\langle \pi(-a^\ast a)v,v \rangle \ge 0$ for every $\cR$-point $(\pi,v)$ such that
$\langle \pi(-x^\ast x)v,v \rangle \ge 0$ for every $x \in I$. By the GNS construction,
this is equivalent to $f(-a^\ast a) \ge 0$ for every linear functional $f$ on $\chrisA_h$
such that $f(\Sigma_\chrisA) \ge 0$ and $f(-x^\ast x) \ge 0$ for every $x \in I$ or, in other words,
to $-a^\ast a \in (\Sigma_\chrisA-\Sigma_I)^{\vee \vee}=\overline{\Sigma_\chrisA-\Sigma_I}$.
\end{proof}

Further  characterizations of $\sqrt[\cR]{I}$ can be obtained by combining Theorem \ref{nsatz}
with Proposition \ref{nsatz-prop}.

\begin{prop}
\label{nsatz-prop}
Let $\chrisA$ be as above and let $I$ be a left ideal of $\chrisA$ generated by
the set $\{p_\lambda\}_{\lambda \in \Lambda}$. Write $S$ for the set
$\{p_\lambda^\ast p_\lambda\}_{\lambda \in \Lambda}$. Then
\[
\Sigma_\chrisA-\cone(S) \subseteq \Sigma_\chrisA-\Sigma_I
\subseteq \Sigma_\chrisA+(I \cap \chrisA_h) \subseteq (\Sigma_\chrisA+I+I^\ast) \cap \chrisA_h
\]
and
\[
\overline{\Sigma_\chrisA-\cone(S)} = \overline{\Sigma_\chrisA-\Sigma_I}=
\overline{\Sigma_\chrisA+(I \cap \chrisA_h)} = \overline{(\Sigma_\chrisA+I+I^\ast) \cap \chrisA_h}.
\]
\end{prop}

\begin{proof}
Clearly,
$
\cone(S) \subseteq \Sigma_I
\subseteq I \cap \chrisA_h \subseteq (I+I^\ast) \cap \chrisA_h,
$
which implies the claimed inclusions.
To prove the equalities, it suffices to
show that $(\Sigma_\chrisA+I+I^\ast)
\cap \chrisA_h \subseteq \overline{\Sigma_\chrisA-\cone(S)}$.
Take any $x \in (\Sigma_\chrisA+I+I^\ast) \cap \chrisA_h$
 and pick $s \in \Sigma_\chrisA$, $u,v \in I$
such that $x=s+u+v^\ast$.
It follows that
$$x=\frac12(x+x^\ast)=s+\frac12(u+v)+\frac12(u+v)^\ast
=s+w+w^\ast$$
where $w=\frac12(u+v) \in I
.$
By the definition of generators, there exists a finite
subset $M$ of $\Lambda$ and elements $q_\mu \in \chrisA$, $\mu \in M$, such that
$w=\sum_{\mu \in M} q_\mu p_\mu$.
For every $\eps>0$, we have that
$$x+\eps \sum q_\mu q_\mu^\ast=s+\sum_{\mu \in M} q_\mu p_\mu
+\sum_{\mu \in M} p_\mu^\ast q_\mu^\ast +\eps \sum q_\mu q_\mu^\ast
$$
$$
=s+\frac{1}{\eps} \sum_{\mu \in M} (p_\mu+ \eps q_\mu^\ast)^\ast(p_\mu
+ \eps q_\mu^\ast)
-\frac{1}{\eps} \sum_{\mu \in M} p_\mu^\ast p_\mu \in \Sigma-\cone(S)
.$$
It follows that $x \in \overline{\Sigma_\chrisA-\cone(S)}$.
\end{proof}



   Corollary \ref{nsatz-cor2} bears some resemblance to  Theorem 7 in
   \cite{hmp}.   The closure in the finest locally convex topology,
   replaces the approximation and archimedean term appearing in that
   Theorem.

\begin{cor}
\label{nsatz-cor2}
For every left ideal $I$ of $\chrisA$
$$
\sqrt[\cR]{I}
=\{a \in \chrisA \mid -a^\ast a \in \overline{(\Sigma_\chrisA+I+I^\ast) \cap \chrisA_h}\}.
$$
\end{cor}

Worth mentioning is also
\begin{cor}
\label{nsatz-cor1}
Suppose that $\{p_\lambda\}_{\lambda \in \Lambda}$ is a subset of $\chrisA$.
If $a \in \chrisA$ satisfies $\pi(a)v=0$ for every $\cR$-point $(\pi,v)$ of $\chrisA$ such
that $ \pi(p_\lambda)v= 0$ for all $\lambda \in \Lambda$, then
$-a^\ast a \in \overline{\Sigma_\chrisA-\cone(S)}$
where $S=\{p_\lambda^\ast p_\lambda\}_{\lambda \in \Lambda}$.
\end{cor}

\subsection{An Auxiliary ``Radical'' $\sqrt[\alpha]{I}$}

Corollary \ref{nsatz-cor2} suggests that for every left ideal $I$ of a $\ast$-algebra $\chrisA$,
the following set is relevant:
\[
\sqrt[\alpha]{I} := \{ a \in \chrisA \mid -a^\ast a \in \Sigma_\chrisA+I+I^\ast\}.
\]
Note that $\sqrt[\alpha]{I} \subseteq \rr{I}$ by the definition of a real ideal.

The remainder of this section is
devoted to a discussion of when $\sqrt[\alpha]{I}$ is an ideal.
%
The next example
shows that it need not be, 
even for a principal left ideal in a
free $\ast$-algebra.
 
 \begin{exa}
  \label{notqr}
  Let $I \subset \FA=\FNSx$ be the left ideal generated by the polynomial $x_1^*x_1$.
  Clearly, $x_1 \in \sqrt[\alpha]{I}$. We claim that $x_1^2 \not\in \sqrt[\alpha]{I}$.

  If $x_1^2 \in \sqrt[\alpha]{I}$, then $(x_1^2)^\ast x_1^2 +\sigma \in I+I^\ast$ for some $\sigma \in \Sigma_\FA$.
  By part (2) of Proposition \ref{prop:lowdeg}, we get
  $x_1^2\in I \oplus \FA_1$,
  which is not possible.
  \qed
 \end{exa}

If the set $(\Sigma_\chrisA+I+I^\ast) \cap \chrisA_h$ is closed, then 
$\sqrt[\alpha]{I} = \RradI$ by Corollary \ref{nsatz-cor2}. It follows that
the set $\sqrt[\alpha]{I}$ is a left ideal and $\RradI = \rr{I}$.

There exists
a large class of $\ast$-algebras in which $\sqrt[\alpha]{I}$ is always a left ideal.
We say that a $\ast$-algebra $\chrisA$ is \textit{centrally bounded} if for every $a \in \chrisA$, there exists
an element $c$ in the center of $\chrisA$ such that $c^\ast c - a^\ast a \in \Sigma_\chrisA$.

\begin{lemma}
\label{lem:cb}
If $I$ is a left ideal of an centrally bounded $\ast$-algebra $\chrisA$ then
the set $\sqrt[\alpha]{I}$ is also a left ideal of $\chrisA$.
\end{lemma}

\begin{proof}
Suppose that $a,b \in \sqrt[\alpha]{I}$. Hence, $-a^\ast a,-b^\ast b \in \Sigma_\chrisA+I+I^\ast$ by the definition of
$\sqrt[\alpha]{I}$. It follows that
$$-(a+b)^\ast(a+b)=(a-b)^\ast(a-b)+2(-a^\ast a)+2(-b^\ast b) \in \Sigma_\chrisA+I+I^\ast. $$
Therefore, $a+b \in \sqrt[\alpha]{I}$. Suppose now that $a \in \chrisA$ and $b \in \sqrt[\alpha]{I}$. Since $\chrisA$ is
centrally bounded, there exists $c$ in the center of $\chrisA$ such that $c^\ast c-a^\ast a \in \Sigma_\chrisA$.
Since $-b^\ast b \in \Sigma_\chrisA+I+I^\ast$, it follows that
$$-b^\ast a^\ast a b = c^\ast c(- b^\ast  b) +
b^\ast (c^\ast c-a^\ast a) b \in \Sigma_\chrisA+I+I^\ast.$$ Therefore $ab \in \sqrt[\alpha]{I}$.
\end{proof}

Clearly, every commutative unital algebra in centrally bounded as well as every algebraically bounded $\ast$-algebra
(in particular, every Banach $\ast$-algebra and every group algebra with standard involution $g^\ast=g^{-1}$).
We would like to show that algebras of matrix polynomials are also
centrally bounded. This follows from the following observation.

\begin{lemma}
If $\chrisA$ is a centrally bounded $\ast$-algebra, then $M_n(\chrisA)$ is also a centrally bounded $\ast$-algebra for every $n$.
\end{lemma}

\begin{proof}
Every element $P \in M_n(\chrisA)$ can be written as $P=\sum_{i,j=1}^n p_{ij} E_{ij}$
where $E_{ij}$ are matrix units. Since $I-E_{ij}^\ast E_{ij}=I-E_{jj}=\sum_{i \ne j} E_{ii}
=\sum_{i \ne j} E_{ii}^\ast E_{ii}$, all matrix units are centrally bounded. By assumption,
elements $p_{ij}I$ are also centrally bounded. Therefore it suffices to show that a
sum and a product of two centrally bounded elements is a centrally bounded element.
Suppose that $c_i^\ast c_i - P_i^\ast P_i \in \Sigma_\chrisA$ for $i=1,2$ where $c_i$
are central and $P_i$ are arbitrary elements of $\chrisA$.
It follows that
\[
\begin{array}{c}
(1+c_1^\ast c_1+c_2^\ast c_2)^2 - (P_1+P_2)^\ast (P_1+P_2) = \\[2mm]
=1+(c_1^\ast c_1+c_2^\ast c_2)^2+2\sum_{i=1}^2(c_i^\ast c_i - P_i^\ast P_i)
+(P_1-P_2)^\ast (P_1-P_2) \in \Sigma_\chrisA
\end{array}
\]
and
\[
\begin{array}{c}
(c_1 c_2)^\ast (c_1 c_2)-(P_1 P_2)^\ast (P_1 P_2)= \\[2mm]
= P_2^\ast (c_1^\ast c_1-P_1^\ast P_1) P_2
+c_1^\ast (c_2^\ast c_2-P_2^\ast P_2) c_1\in \Sigma_\chrisA.
\end{array}
\]
\end{proof}

\subsection{An Iterative Description of $\rr{I}$}
  \label{sec:cimNssReI}

For a left ideal $I$ in a $\ast$-algebra $\chrisA$, let
$\sqrt[\beta]{I}$ denote the left ideal in $\chrisA$ generated by
 $\sqrt[\alpha]{I}$;  i.e.
\[
\sqrt[\beta]{I} := \chrisA \sqrt[\alpha]{I}.
\]

Unlike the {\qrr}, $\sqrt[\beta]{\cdot }$ is not
idempotent. However, we do have the following:

\begin{prop}
 \label{prop:Jaka-procedure}
If $I$ is a left ideal of a $\ast$-algebra $\chrisA,$ then
\[
\sqrt[\beta]{I} \cup \sqrt[\beta]{\sqrt[\beta]{I}} \cup \sqrt[\beta]{\sqrt[\beta]{\sqrt[\beta]{I}}} \cup \ldots =
\rr{I}.
\]
\end{prop}

\begin{proof}
Write $I_0=I$ and $I_{n+1}=\sqrt[\beta]{I_n}$ for every $n=0,1,2,\ldots$.
Hence, the left-hand side of the formula is $J:=\bigcup_{n=0}^\infty I_n$.
To show that $J \subseteq \rr{I}$, it suffices to show that
$I_n \subseteq \rr{I}$ for every $n$. This is clear for $n=0$.
Suppose this is true for some $n$ and pick $x \in I_{n+1}$. By the definition
of $I_{n+1}$, $x=\sum_{i=1}^k a_i y_i$, where $a_i \in \chrisA$ and
$-y_i^\ast y_i \in \Sigma_\chrisA+I_n+I_n^\ast$ for $i=1,\ldots,k$.
Since $I_n \subseteq \rr{I}$ and $\rr{I}$ is real,
it follows that $y_i \in \rr{I}$ for every $i=1,\ldots,k$. Hence $x \in \rr{I}$.
We will prove the opposite inclusion $\rr{I} \subseteq J$ by showing that $J$ is real.
Pick $u_1,\ldots,u_r \in  \chrisA$ such that $\sum_{i=1}^r u_i^\ast u_i \in J+J^\ast$. By the definition of $J$,
there exists a number $n$ and elements $b,c \in I_n$ such that $\sum_{i=1}^r u_i^\ast u_i=b+c^\ast$.
It follows that for every $i=1,\ldots,r$, $-u_i^\ast u_i \in \Sigma_\chrisA+I_n+I_n^\ast$.
Therefore $u_i \in \sqrt[\alpha]{I_n} \subseteq \sqrt[\beta]{I_n}=I_{n+1} \subseteq J$.
\end{proof}

  Specializing the iterative procedure of Proposition
  \ref{prop:Jaka-procedure}, which works in all $\ast$-algebras,
 to the case of a left ideal in free
  $\ast$-algebra does not lead to the Real Algorithm. Here is an informal comparison:

\begin{enumerate}
\item
Proposition \ref{prop:Jaka-procedure} adds
all  tuples $(q_i)$ such that $\sum_i q_i^\ast q_i \in I_k+I_k^\ast$
to $I_k$ to produce the update $I_{k+1}$; whereas
the Real Algorithm adds one such tuple $(q_i)$ which was well chosen to $I^{(k)}$ to
produce $I^{(k+1)}$.
\item
  For a general $\ast$-algebra $\chrisA$ and left ideal $I$,
  the iterations in Proposition \ref{prop:Jaka-procedure}
  do not necessarily  stop unless $\chrisA$ is left noetherian
 (such us $M_n(F[x])$, see \S \ref{sec:examples}.) 
However, in the case $I$ is a left
ideal in the free $\ast$-algebra $\FA$, the inclusion
 sense for finitely generated left ideals in $I^{(k)} \subseteq I_k$
 implies the procedure of Proposition \ref{prop:Jaka-procedure} does
 terminate.
\item
  Unlike the Real Algorithm, even if only finitely many iterations are needed in Proposition
  \ref{prop:Jaka-procedure},
  it does not tell us how to obtain generators
 of $\rr{I}$ from the generators of $I$. (This is a nontrivial problem
 even for $\RR[x]$, cf. \cite{llr} for a partial
 solution, and it is still open for $M_n(F[x])$.)
 \end{enumerate}


For centrally bounded algebras, Proposition \ref{prop:Jaka-procedure} and Lemma \ref{lem:cb} imply 
the following simple iterative description of the elements of the {\qrr}:

\begin{cor}
\label{cor:simpdescr}
Let $I$ be a left ideal of a centrally bounded $\ast$-algebra $\chrisA$. An element
$x \in \chrisA$ belongs to $\rr{I}$ if there exist $m \in \NN$,
$s_1,\ldots,s_m \in \Sigma_\chrisA$ and $k_1,\ldots,k_m \in \{a \in \chrisA \mid a^\ast=-a\}$
such that the last term of the sequence \[x_1:=x, \quad x_{i+1}:=x_i^\ast x_i+s_i+k_i,
i=1,\ldots,m, \]
belongs to $I$.
\end{cor}

For commutative $\ast$-algebras, we have the following
classical real Nullstellensatz:

\begin{cor}
\label{prop:srr}
For every ideal $I$ of a commutative $\ast$-algebra $\chrisA$ we have that
\begin{equation*}\begin{split}
\rr{I}  &= \{a \in \chrisA \mid -(a^\ast a)^k \in \Sigma_\chrisA+I+I^\ast \text{ for some } k\} \\
&= \{a \in \chrisA \mid -(a^\ast a)^k \in \Sigma_\chrisA+I \text{ for some } k\}.
\end{split}\end{equation*}
\end{cor}

\begin{proof}
For every ideal $J$ of $\chrisA$ write 
\begin{equation*}
\sqrt[\gamma]{J} := \{ a \in \chrisA \mid -a^\ast a \in \Sigma_\chrisA+J\}.
\end{equation*}
Since $J \subseteq \sqrt[\gamma]{J}$, $(\sqrt[\gamma]{J})^\ast=\sqrt[\gamma]{J}$ and  $\sqrt[\gamma]{J+J^\ast} = \sqrt[\alpha]{J}$, 
we have that 
\begin{equation}
\label{gam1}
\sqrt[\gamma]{J} \subseteq \sqrt[\alpha]{J} \subseteq \sqrt[\gamma]{\sqrt[\gamma]{J}}.
\end{equation}
If $a \in \sqrt[\gamma]{\sqrt[\gamma]{J}}$ for some $a \in \chrisA$, then $a^\ast a+\sigma \in \sqrt[\gamma]{J}$
for some $\sigma \in \Sigma_\chrisA$. It follows that $(a^\ast a+\sigma)^2+\tau \in J$ for some $\tau\in \Sigma_\chrisA$.
Since $2 \sigma a^\ast a+\sigma^2+\tau \in  \Sigma_\chrisA$, it follows that $a^\ast a \in \sqrt[\gamma]{J}$. Therefore
\begin{equation}
\label{gam2}
\sqrt[\gamma]{\sqrt[\gamma]{J}}=\{a \in \chrisA \mid a^\ast a \in \sqrt[\gamma]{J}\}.
\end{equation}

For every ideal $I$ of $\chrisA$ we define two sequences:
$$I_0=I, I_{n+1}=\sqrt[\alpha]{I_n} \quad \mbox{ and } \quad
K_0=I, K_{n+1} = \sqrt[\gamma]{K_n}.$$
 By induction on $n$, using (\ref{gam1}),
we show that $K_n \subseteq I_n \subseteq K_{2^n}$. 
By Proposition \ref{prop:Jaka-procedure}, it follows that
\begin{equation}
\label{gam3}
\bigcup_{n=0}^\infty K_n=\bigcup_{n=0}^\infty I_n=\rr{I}.
\end{equation}
Note also that $\sqrt[\alpha]{I}=\sqrt[\alpha]{I+I^\ast}$, hence $$\rr{I}=\rr{I+I^\ast}$$ by Proposition \ref{prop:Jaka-procedure}.
On the other hand, equation (\ref{gam2}) implies that 
\begin{equation}
K_n=\{a \in \chrisA \mid -(a^\ast a)^{2^{n-1}} \in \Sigma_\chrisA+I\}.
\end{equation}
To finish the proof, note that $-(a^\ast a)^n \in \Sigma_\chrisA+I$ implies $-(a^\ast a)^{2^{n-1}} \in \Sigma_\chrisA+I$.
\end{proof}

\section{A Nullstellensatz for $M_n(F[x])$}
\label{sec:examples}

We will discuss the following question:

\bs

\noindent
{\bf Question:}
Which  left ideals $I$ in $M_n(F[x])$ satisfy
$\sqrt[\cE]{I}=\rr{I}$? 

\bs

Recall that $\sqrt[\cE]{I}=\{Q \in M_n(F[x]) \mid Q(a)v=0$ for every $a \in \RR^g$ and $v \in F^n$ such that
$P(a)v=0$ for all $P(x) \in I\}$.

We will prove the answer is yes for all
$I$ 
 in the cases of $g=0$ and $g=1$ variables,
see Propositions \ref{g0thm} and \ref{g1thm}. 
The case of several variables remains undecided,
except for $n=1$ which is classical, see Example \ref{ex:classical}

Example \ref{ex:classical} rephrases the classical Real Nullstellensatz
of Dubois \cite{dub}, Risler \cite{ris} and Efroymson \cite{efr},
and extends it from $\RR[x]$ to $\CC[x]$. 

\begin{exa}
\label{ex:classical}
For every ideal $I$ of $F[x]$ we have that
$$
\sqrt[\cE]{I}= \rr{I}.
$$

If a polynomial $q \in F[x]$ belongs to $\sqrt[\cE]{I}$,
then $q(a)v=0$ for every $(a,v) \in \RR^g \times F$ such that $p(a)v=0$ for all $p \in I$.
It follows that $q(a)=0$ for every $a \in \RR^g$ such that $p(a)=0$ for all $p \in I$,
hence $(\bar{q}q)(a)=0$ for every $a \in \RR^g$ such that $(\bar{p}p)(a)=0$ for all $p \in I$.
By the classical Real Nullstellensatz, there exists $k \in \NN$
such that $-(\bar{q}q)^{2k} \in \Sigma_\chrisA+$ ideal generated by $\bar{p}p$, $p \in I$.
It follows that $q \in \rr{I}$.
 \qed
\end{exa}

\begin{prop}
\label{g0thm}
For every left ideal $I$ of $M_n(F)$, we have that
$$
I=\rr{I} = \sqrt[\cE]{I}.
$$
\end{prop}

\begin{proof}
It suffices to show that $\sqrt[\cE]{I} \subseteq I$.
Since $M_n(F)$ is finite-dimensional, $I$ is finitely generated, let $B_1,\ldots,B_r$ be the generators of $I$
as a left ideal. It follows that
\[
\sqrt[\cE]{I}=\{C \in M_n(F) \mid \ker \mathbf{B} \subseteq \ker C\}
\quad \text{ where } \quad \mathbf{B}=\left[ \begin{array}{c} B_1 \\ \vdots \\ B_r \end{array} \right].
\]
For each $C \in \sqrt[\cE]{I}$, one sees that $\ker \mathbf{B} \subseteq \ker C$, which implies that the row space of $C$ is contained in the row space of $\mathbf{B}$. Therefore, there exists a matrix $R=\left[R_1 \ldots R_r \right]$ such that
$C=R\mathbf{B}$. It follows that $C \in I$.
\end{proof}

\begin{thm}
\label{g1thm}
For every positive integer $n$ and every left ideal $I$ in $M_n(F[x_1])$ we have that
$$\sqrt[\cE]{I}=\rr{I}.$$
\end{thm}


\begin{proof} The proof consists of three steps:
\begin{enumerate}
\item Reduction to the case $I=(P)$, that is, the case where $I$ is a principal ideal.
\item Reduction to the case where $P$ is diagonal.
\item Induction on $n$.
\end{enumerate}
Steps (1) and (3) also work for several variables but step (2) does not.

\medskip

Since $F[x_1]$ is left noetherian so is $M_n(F[x_1])$, see Proposition 1.2. in \cite{mcr}.
Therefore $I=(P_1,\ldots,P_k)$ for some $P_1,\ldots,P_k \in M_n(F[x_1])$.
Define $P=P_1^\ast P_1+\ldots+P_k^\ast P_k$ and note that $(P) \subseteq I \subseteq \sqrt[\alpha]{(P)}$.
It follows that $\rr{I}=\rr{(P)}$ and $\sqrt[\cE]{I}=\sqrt[\cE]{(P)}$,
proving (1).

\medskip

Let $P=UDV$ be the Smith normal form of $P$, i.e. $U$ and $V$ are invertible in $M_n(F[x_1])$
and $D$ is diagonal. Since $(P)=(DV)$, it suffices to prove that
$\sqrt[\cE]{(DV)}=\sqrt[\cE]{(D)} V$ and $\rr{(DV)}=\rr{(D)}V$.
Clearly, $R \in \sqrt[\cE]{(DV)}$ iff $R(a)w=0$ for every $a \in \RR$ and $w \in F^n$ such that $D(a)V(a)w=0$
iff $R(a)V(a)^{-1}z=0$ for every $a \in \RR$ and $z \in F^n$ such that $D(a)z=0$ iff $RV^{-1} \in \sqrt[\cE]{(D)}$.
To prove the second equality, it suffices to show that $\rr{(DV)} \subseteq \rr{(D)}V$.
Namely, replacing $V$ by $V^{-1}$ and $D$ by $DV$, we get the opposite inclusion.
We have to show that the left ideal $\rr{(D)}V$, which contains $(DV)$, is real.
Suppose that $\sum_i Q_i^\ast Q_i \in \rr{(D)}V$ for some $Q_i$. It follows that
$\sum_i (V^{-1})^\ast Q_i^\ast Q_i V^{-1} \in (V^{-1})^\ast \rr{(D)} \subseteq\rr{(D)}$,
hence $Q_i V^{-1} \in \rr{(D)}$ for all $i$.

\medskip

We will show now that $\sqrt[\cE]{(D)}=\rr{(D)}$ by induction on $n$.
For $n=1$ this is Example \ref{ex:classical}. Now we assume that
$\sqrt[\cE]{(D_1)} \subseteq \rr{(D_1)}$ and $\sqrt[\cE]{(D_2)} \subseteq \rr{(D_2)}$
and claim that $\sqrt[\cE]{(D_1 \oplus D_2)} \subseteq \rr{(D_1 \oplus D_2)}$.
Pick any $R =\left[ R_1 \ R_2 \right] \in \sqrt[\cE]{(D_1 \oplus D_2)}$.
>From the definition of $\sqrt[\cE]{\cdot}$ we get that $R_1(a)v_1+R_2(a)v_2=0$ for every $a \in \RR$, $v_1 \in F^{n_1}$
and $v_2 \in F^{n_2}$ such that $D_1(a)v_1=0$ and $D_2(a)v_2=0$. Inserting either $v_2=0$ or $v_1=0$
we get (for each $i$) that $R_i(a)v_i=0$ for every $a \in \RR$ and $v_i \in F^{n_i}$ such that $D_i(a)v_i=0$.
Note that $R_i(a)v_i=0$ implies $R_i(a)^\ast R_i(a)v_i=0$ and that $R_i^\ast R_i$ is a square matrix of size $n_i$.
It follows that $R_i^\ast R_i \in \sqrt[\cE]{(D_i)} \subseteq \rr{(D_i)}$.
Let $j_i \colon M_{n_i}(F[x_1]) \to M_{n_1+n_2}(F[x_1])$ be the natural embeddings.
Since $j_i$ are $\ast$-homomorphisms and $J_i=\rr{(j_i(D_i))}$ are real left ideals,
$j_i^{-1}(J_i)$ are also real left ideals,
so that $\rr{(D_i)} \subseteq j_i^{-1}(J_i)$.
Since $j_i(D_i)$ is the product of $j_i(I_{n_i})$ and $D_1 \oplus D_2$, it belongs to $(D_1 \oplus D_2)$.
Hence, for $i=1,2$,
\[
j_i(R_i^\ast R_i) \in
j_i(\rr{(D_i)}) \subseteq
\rr{(j_i(D_i))} \subseteq \rr{(D_1 \oplus D_2)}.
\]
Since $\left[ R_1 \ 0 \right]^\ast \left[ R_1 \ 0 \right] = j_1(R_1^\ast R_1)$
and $\left[ 0 \ R_2 \right]^\ast \left[ 0 \ R_2 \right] = j_2(R_2^\ast R_2)$
belong to $\rr{(D_1 \oplus D_2)}$,
$\left[ R_1 \ 0 \right]$ and $\left[ 0\ R_2  \right]$ also belong to
$\rr{(D_1 \oplus D_2)}$. Therefore, $\left[ R_1 \ R_2 \right]=
\left[ R_1 \ 0 \right]+\left[ 0\ R_2  \right] \in \rr{(D_1 \oplus D_2)}$.
\end{proof}

\end{document}